\documentclass[a4paper, 10pt, notitlepage]{article}

\usepackage{amsthm, amsmath, amssymb, latexsym, eufrak}
\usepackage{mathrsfs,cite}
\usepackage[multiple]{footmisc}
\usepackage{graphicx}
\usepackage[ruled,vlined]{algorithm2e}

\theoremstyle{plain}
\newtheorem{theorem}{Theorem}
\newtheorem{lemma}{Lemma}
\newtheorem{corollary}{Corollary}

\theoremstyle{definition}
\newtheorem{definition}{Definition}

\newtheorem{assumption}{Assumption}

\theoremstyle{remark}
\newtheorem{remark}{Remark}

\DeclareMathOperator{\dist}{dist}
\DeclareMathOperator{\dom}{dom}

\author{M.V. Dolgopolik\footnote{Institute for Problems in Mechanical Engineering of the Russian Academy of Sciences,
Saint Petersburg, Russia}}
\title{A universal convergence theorem for primal-dual penalty and augmented Lagrangian methods}

\begin{document}

\maketitle

\begin{abstract}
We present a so-called universal convergence theorem for inexact primal-dual penalty and augmented Lagrangian methods
that can be applied to a large number of such methods and reduces their convergence analysis to verification of some 
simple conditions on sequences generated by these methods. If these conditions are verified, then both primal and dual
convergence follow directly from the universal convergence theorem. This theorem allows one not only to derive standard
convergence theorems for many existing primal-dual penalty and augmented Lagrangian methods in a unified and
straightforward manner, but also to strengthen and generalize some of these theorems. In particular, we show how with
the use of the universal convergence theorem one can significantly improve some existing results on convergence of 
a primal-dual rounded weighted $\ell_1$-penalty method, an augmented Lagrangian method for cone constrained
optimization, and some other primal-dual methods.
\end{abstract}

\section{Introduction}

Design and convergence analysis of primal-dual penalty and augmented Lagrangian methods have been active areas of
research in constrained optimization for many years
\cite{Bertsekas,LuoSunLi,LuoSunWu,WangLi2009,Burachick2011,LuoWuChen2012}. Such methods, apart from being popular 
algorithms for solving various convex programming problems
\cite{Rockafellar73,TsengBertsekas,Iusem99,Polyak2002,Xu2021,CuiDingLiZhao}, have also found applications in
constrained global optimization \cite{BirginFloudasMartinez,CordovaOliveiraSagastizabal}. Furthermore, in the nonconvex
case primal-dual methods combined with local search unconstrained optimization algorithms instead of the global ones
sometimes significantly outperform specialized nonlinear programming solvers in computation time and/or quality of
computed local minimizers (see the results of numerical experiments in \cite{BurachikKayaPrice,Dolgopolik_DCSemidef}).

Although convergence analysis of primal-dual penalty and augmented Lagrangian methods is a well-established topic in
constrained optimization, little to no attention has been paid to uncovering general principles on which such analysis
is based. If one carefully reads numerous papers dealing with convergence of primal-dual methods, one will quickly
notice that very similar arguments are used in all these paper and very similar results are proved again and again for
each new primal-dual method and each new type of constrained optimization problems (for example, compare convergence
analysis from \cite{BirginMartinez,CordovaOliveiraSagastizabal,BurachikKayaPrice,LuoWuChen2012,WuLuoDingChen2013}).
Some attempts to unify convergence analysis of augmented Lagrangian methods for inequality constrained problems were
made in \cite{LuoSunLi,LuoSunWu,WangLi2009}. However, even in this specific case no real unification was achieved,
since neither of the results of one of these papers can be derived from the results of the other two due to the use of
different augmented Lagrangians and slightly different, non-interchangable assumptions, despite the fact that papers
\cite{LuoSunLi,LuoSunWu,WangLi2009} contain very similar results that are based on very similar proofs.

The main goal of this paper is to present a so-called \textit{universal convergence theorem} for primal-dual penalty
and augmented Lagrangian methods that (i) uncovers general principles on which convergence analysis of such methods is
based and (ii) can be applied to the vast majority of existing primal-dual methods for various types of constrained
optimization problems, including nonlinear programming problems, cone constrained optimization problems, etc. To this
end, we present a general duality scheme and an inexact primal-dual method based on this scheme that encompasses many
existing primal-dual penalty and augmented Lagrangian methods as particular cases. Then we prove a general convergence
theorem for this inexact primal-dual method that we call the universal convergence theorem. This theorem reduces the
proof of primal and dual convergence to verification of some simple conditions on sequences generated by the method, the
main one of which is the convergence of the infeasibility measure to zero. Verification of these conditions is at the 
core of the proofs of most existing convergence theorems for primal-dual methods. When these conditions are verified,
the same standard arguments are used to prove primal and/or dual convergence. The universal convergence theorem
eliminates the need to repeat these arguments.

In order to demonstrate the strengths of the universal convergence theorem and underline its universal nature, we
consider its applications to convergence analysis of four different primal-dual penalty/augmented Lagrangian methods:
(i) a simple primal-dual penalty method, (ii) a primal-dual rounded weighted $\ell_1$-penalty method proposed in
\cite{BurachikKayaPrice}, (iii) an extension of the classical augmented Lagrangian method based on the
Hestenes-Powell-Rockafellar augmented Lagrangian \cite{BirginMartinez} to the case of general cone constrained
optimization problems and, finally, (iv) a multiplier algorithm based on the so-called P-type Augmented Lagrangian
Method (PALM) from \cite{WangLi2009}. We show that with the use of the universal convergence theorem one can not only
easily recover existing convergence theorems for these methods, but also significantly strengthen and generalize them. 

The paper is organized as follows. A general duality scheme that allows one to describe many existing primal-dual 
penalty and augmented Lagrangian methods in a unified manner is presented in Section~\ref{sect:DualityScheme}. 
Section~\ref{sect:UnivConvTheorem} is devoted to the universal convergence theorem for inexact primal-dual methods, 
while Section~\ref{sect:Applications} contains applications of this theorem to convergence analysis of several different
penalty and augmented Lagrangian methods. Namely, in Subsection~\ref{subsect:PenaltyMethod} we consider a simple
primal-dual penalty method, in Subsection~\ref{subsect:RoundedWeightedPenaltyMethod} a primal-dual rounded weighted
$\ell_1$-penalty method, in Subsection~\ref{subsect:AugmLagrMethod} an augmented Lagrangian method for general cone
constrained optimization problems, and in Subsection~\ref{subsect:PALMAlgorithm} a multiplier algorithm based on PALM.

\section{General duality scheme}
\label{sect:DualityScheme}

Let $(X, d_X)$ be a metric space, $Y$ be a normed space, $Q \subset X$ be a nonempty closed set, 
$f \colon X \to \mathbb{R} \cup \{ + \infty \}$ be a given function, and $\Phi \colon X \rightrightarrows Y$ be some
multifunction. Throughout the first part of this article we study primal-dual methods for solving the following
constrained optimization problem:
\[
  \min \: f(x) \quad \text{subject to} \quad 0 \in \Phi(x), \quad x \in Q.
  \qquad \eqno{(\mathcal{P})}
\]
We assume that the feasible region of this problem is nonempty and its optimal value, denoted by $f_*$, is finite. 

Let us introduce a general duality scheme for the problem $(\mathcal{P})$ that can be used to derive and analyse
many existing primal-dual optimization methods in a unified manner. Let $(M, d_M)$ be a metric space of parameters and
suppose that some function $\varphi \colon X \times M \to (- \infty, + \infty]$ is given. We \textit{formally} define 
\textit{a merit function} for the problem $(\mathcal{P})$ as follows:
\[
  F(x, \mu) := f(x) + \varphi(x, \mu) \quad x \in X, \: \mu \in M.
\]
Meanings of the function $\varphi$ and parameter $\mu$ depend on the context. If one wants to consider a penalty
method, then $\mu \ge 0$ is a penalty parameter and one can define, for example, 
$\varphi(x, \mu) = \mu \dist(0, \Phi(x))$. If one consideres an augmented Lagrangian method, then one can set 
$\mu = (\lambda, c)$, where $\lambda$ is a multiplier and $c$ is a penalty parameter. In this case the function
$\varphi(x, \mu)$ can be defined in many different ways. For example, if $\Phi$ is single-valued and $Y$ is a Hilbert
space, then one can define $\varphi(x, \mu) = \langle \lambda, \Phi(x) \rangle + 0.5 c \| \Phi(x) \|^2$, where 
$\mu = (\lambda, c)$, $\lambda \in Y$, and $\langle \cdot, \cdot \rangle$ is the inner product. 

In the general case, $\mu$ is simply a parameter of the merit function $F$, while 
$\varphi \colon X \times M \to (- \infty, + \infty]$ is a given function. No assumptions on this function are needed to
\textit{formally} define a duality scheme and a primal-dual method based on the merit function $F$. Some assumptions on
the function $\varphi$ will appear later as a part of convergence analysis of this method.

\textit{The dual problem} (associated with the merit function $F$) is defined in the following way:
\[
  \max_{\mu \in M} \: \Theta(\mu), \qquad \eqno{(\mathcal{D})}
\]
where
\[
  \Theta(\mu) := \inf_{x \in Q} F(x, \mu), \quad \mu \in M
\]
is \textit{the dual function} (associated with the merit function $F$). The optimal value of the dual problem is denoted
by $\Theta_*$. The following standard definition describes how the optimal values of the primal and dual problems should
be related to each other in order to make primal-dual methods based on the proposed duality scheme tractable.

\begin{definition}
One says that \textit{the weak duality} holds (for the pair of problems $(\mathcal{P})$ and $(\mathcal{D})$), if
$\Theta_* \le f_*$. One says that \textit{the strong duality} holds (or that \textit{the zero duality gap property}
holds true for the merit function $F$), if $\Theta_* = f_*$.
\end{definition}

Let us point out a natural assumption that ensures the weak duality.

\begin{lemma} \label{lem:WeakDuality}
Suppose that for any $\mu \in M$ and any feasible point $x$ of the problem $(\mathcal{P})$ one has 
$\varphi(x, \mu) \le 0$. Then the weak duality holds for the pair of problems $(\mathcal{P})$ and $(\mathcal{D})$.
\end{lemma}

\begin{proof}
Denote by $\Omega$ the feasible region of the problem $(\mathcal{P})$. By our assumption for any $x \in \Omega$ and 
$\mu \in M$ one has $F(x, \mu) \le f(x)$. Therefore
\[
  \Theta(\mu) := \inf_{x \in Q} F(x, \mu) \le \inf_{x \in \Omega} F(x, \mu) \le \inf_{x \in \Omega} f(x) = f_*
  \quad \forall \mu \in M,
\] 
which implies that $\Theta_* \le f_*$.
\end{proof}

\section{Inexact primal-dual method and its convergence analysis}
\label{sect:UnivConvTheorem}

Let us now turn to analysis of primal-dual methods for solving the problem $(\mathcal{P})$ that are based on the merit
function $F$. A large class of existing primal-dual penalty and augmented Lagrangian methods can be described with the
use of the following \textit{inexact primal-dual method} given in Algorithm~\ref{alg:PrimalDualMethod}. 

\begin{algorithm}
\caption{Inexact primal-dual method}

\textbf{Initialization.} Choose an initial value of the parameter $\mu_0 \in M$ and a bounded above sequence of
tolerances $\{ \varepsilon_n \} \subset [0, + \infty)$. Put $n = 0$.

\textbf{Step 1. Solution of subproblem.} Find an $\varepsilon_n$-optimal solution $x_n$ of the problem
\[
  \min_x \: F(x, \mu_n) \quad \text{subject to} \quad x \in Q,
\]
that is, find $x_n \in Q$ such that $F(x_n, \mu_n) \le F(x, \mu_n) + \varepsilon_n$ for all $x \in Q$.

\textbf{Step 2. Update of parameters.} Choose some $\mu_{n + 1} \in M$, increment $n$, and go to
\textbf{Step 1}.
\label{alg:PrimalDualMethod}
\end{algorithm}

Throughout this article we impose the following assumption on sequences generated by primal-dual methods that is needed 
to ensure that the primal sequence $\{ x_n \}$ is correctly defined.

\begin{assumption} \label{as:PrimalSequenceWellDefined}
For any $n \in \mathbb{N}$ the function $F(\cdot, \mu_n)$ is bounded below on $Q$.
\end{assumption}

It should be noted that we do not assume that $\varepsilon_n \to 0$ as $n \to \infty$, that is, we do not assume that
an approximate solution of the subproblem
\begin{equation} \label{eq:MeritFuncMinProblem}
  \min_x \: F(x, \mu_n) \quad \text{subject to} \quad x \in Q
\end{equation}
can be computed with arbitrarily high precision, as it is often done in the literature (see, e.g.
\cite{LuoSunLi,WangLi2009,BurachickIusemMelo_SharpLagr,Burachick2011,BurachikIusemMelo2013}). The value
$\varepsilon_n > 0$ represents not how accurately one \textit{needs} to solve problem \eqref{eq:MeritFuncMinProblem},
but rather how accurately one \textit{can} solve it in a reasonable amount of time. The assumption that 
$\varepsilon_n \to 0$ as $n \to \infty$ is purely theoretical in nature and is needed only for better general
understanding of asymptotic behaviour of sequences generated by primal-dual methods. That is why in the case when
$\varepsilon_n \to 0$ as $n \to \infty$ we call Algorithm~\ref{alg:PrimalDualMethod} \textit{the asymptotically exact
primal-dual method}.

\begin{remark}
Throughout this article we are interested only in the convergence analysis of infinite sequences generated by
primal-dual penalty and augmented Lagrangian methods. Therefore, we do not include stopping criteria in the descriptions
of these methods and always assume that a method under consideration generates an infinite sequence of points whose
convergence to (approximately) globally optimal solutions we are aiming to prove. Natural stopping criteria for these
methods can be easily designed using the results of our convergence analysis.
\end{remark}

Our aim is to prove a general convergence theorem for the inexact primal-dual method, which we call \textit{a universal
convergence theorem}. This theorem reduces convergence analysis of primal-dual methods to verification of some natural
conditions. In the following section we will demonstrate that in many particular cases these conditions can be verified
in a straightforward manner.

The fundamental assumption that one must impose on primal-dual methods is the decay of the infeasibility measure. 
If the infeasibility measure does not converge to zero, then the corresponding method fails to solve an optimization
problem to which it was applied. In our case, the decay of the infeasibility measure can be expressed as an assumption
that $\dist(0, \Phi(x_n)) \to 0$ as $n \to \infty$, where $\dist(0, \Phi(x_n)) = \inf_{y \in \Phi(x_n)} \| y \|$ is 
the distance between $0$ and $\Phi(x_n)$. This condition along with an additional assumption on the sequence 
$\{ \varphi(x_n, \mu_n) \}$ are, roughly speaking, enough to draw multiple conclusions about the behaviour of 
the sequence $\{ (x_n, \mu_n) \}$ generated by the inexact primal-dual method. To conveniently formulate them, recall
that the function
\[
  \beta(y) = \inf\Big\{ f(x) \Bigm| x \in Q \colon 0 \in \Phi(x) - y \Big\}
\]
is called \textit{the optimal value (perturbation) function} for the problem $(\mathcal{P})$. In addition, denote 
$\varepsilon_* = \limsup_{n \to \infty} \varepsilon_n$. Note that $0 \le \varepsilon_* < + \infty$, since by our
assumption the sequence $\{ \varepsilon_n \}$ is bounded above.

\begin{lemma} \label{lem:InexactPDMLimits}
Let the weak duality hold and $\{ (x_n, \mu_n) \}$ be the sequence generated by the inexact primal-dual method. 
Suppose that 
\begin{equation} \label{eq:BasicAssumptions}
  \lim_{n \to \infty} \dist(0, \Phi(x_n)) = 0, \quad 
  \liminf_{n \to \infty} \varphi(x_n, \mu_n) \ge 0.
\end{equation}
Then the following inequalities hold true:
\begin{align} \label{eq:ObjFuncInexactLim}
  \min\Big\{ f_*, \liminf_{y \to 0} \beta(y) \Big\} \le \liminf_{n \to \infty} f(x_n)
  &\le \limsup_{n \to \infty} f(x_n) \le f_* + \varepsilon_*,
  \\ \label{eq:MeritFuncInexactLim}
  \min\Big\{ f_*, \liminf_{y \to 0} \beta(y) \Big\} \le \liminf_{n \to \infty} F(x_n, \mu_n)
  &\le \limsup_{n \to \infty} F(x_n, \mu_n) \le f_* + \varepsilon_*,
  \\ \label{eq:DualFuncInexactLim}
  \min\Big\{ f_*, \liminf_{y \to 0} \beta(y) \Big\} - \varepsilon_* \le \liminf_{n \to \infty} \Theta(\mu_n)
  &\le \limsup_{n \to \infty} \Theta(\mu_n) \le f_*.
\end{align}
\end{lemma}

\begin{proof}
\textbf{Part 1. Upper estimates.} Since the weak duality holds, by the definition of $x_n$ one has
\[
  f_* \ge \Theta_* \ge \Theta(\mu_n) \ge F(x_n, \mu_n) - \varepsilon_n = f(x_n) + \varphi(x_n, \mu_n) - \varepsilon_n
  \quad \forall n \in \mathbb{N},
\]
which yields
\[
  \limsup_{n \to \infty} \Theta(\mu_n) \le f_*, \quad \limsup_{n \to \infty} F(x_n, \mu_n) \le f_* + \varepsilon_*
\]
and
\begin{align*}
  \limsup_{n \to \infty} f(x_n) &\le \limsup_{n \to \infty} \big( f_* + \varepsilon_n - \varphi(x_n, \mu_n) \big)
  \\
  &\le f_* + \limsup_{n \to \infty} \varepsilon_n + \limsup_{n \to \infty} (- \varphi(x_n, \mu_n))
  \\
  &= f_* + \varepsilon_* - \liminf_{n \to \infty} \varphi(x_n, \mu_n) \le f_* + \varepsilon_*.
\end{align*}
that is, the upper estimates in \eqref{eq:ObjFuncInexactLim}--\eqref{eq:DualFuncInexactLim} are satisfied.

\textbf{Part 2. Lower estimates.} Choose any subsequence $\{ x_{n_k} \}$ such that 
\[
  \lim_{k \to \infty} f(x_{n_k}) = \liminf_{n \to \infty} f(x_n).
\]
At least one such subsequence exists by the definition of lower limit. Let us consider two cases.

\textbf{Case I.} Suppose that there exists a subsequence of the sequence $\{ x_{n_k} \}$, which we denote again by
$\{ x_{n_k} \}$, that is feasible for the problem $(\mathcal{P})$. Then $f(x_{n_k}) \ge f_*$ for all $k \in \mathbb{N}$
and
\[
  \liminf_{n \to \infty} f(x_n) = \lim_{k \to \infty} f(x_{n_k}) \ge f_*.
\]
\textbf{Case II.} Suppose that a feasible subsequence of the sequence $\{ x_{n_k} \}$ does not exist. Then without loss
of generality one can suppose that every element of the sequence $\{ x_{n_k} \}$ is infeasible. By the definition of 
optimal value function for any $k \in \mathbb{N}$ and $y_k \in Y$ such that $0 \in \Phi(x_{n_k}) - y_k$ one has
$f(x_{n_k}) \ge \beta(y_k)$. Recall that $\dist(0, \Phi(x_n)) \to 0$ as $n \to \infty$ by our assumption. Consequently,
one can find a sequence $\{ y_k \} \subset Y$ such that $0 \in \Phi(x_{n_k}) - y_k$ for all $k \in \mathbb{N}$ and 
$\| y_k \| \to 0$ as $k \to \infty$. Therefore
\[
  \liminf_{n \to \infty} f(x_n) = \lim_{k \to \infty} f(x_{n_k}) \ge \liminf_{k \to \infty} \beta(y_k) 
  \ge \liminf_{y \to 0} \beta(y).
\]
Combining the two cases one can conclude that the lower estimate in \eqref{eq:ObjFuncInexactLim} holds true. Bearing in
mind the second condition in \eqref{eq:BasicAssumptions} one gets
\[
  \liminf_{n \to \infty} F(x_n, \mu_n) = \liminf_{n \to \infty} \big( f(x_n) + \varphi(x_n, \mu_n) \big)
  \ge \liminf_{n \to \infty} f(x_n),
\]
that is, the lower estimate in \eqref{eq:MeritFuncInexactLim} is valid. Finally, by the definition of $x_n$ one has
\[
  \liminf_{n \to \infty} \Theta(\mu_n) \ge \liminf_{n \to \infty} (F(x_n, \mu_n) - \varepsilon_n)
  \ge \liminf_{n \to \infty} F(x_n, \mu_n) - \varepsilon_*,
\]
which along with the first inequality in \eqref{eq:MeritFuncInexactLim} implies that the lower estimate in
\eqref{eq:DualFuncInexactLim} holds true.
\end{proof}

One can strengthen the claim of the previous lemma with the use of an additional technical assumption on 
the sequence $\{ \mu_n \}$/the function $\varphi$ that, as we will show below, can be easily verified in many
particular cases.

\begin{definition} \label{def:PhiConvToZero}
A sequence $\{ \mu_n \} \subset M$ is said to satisfy \textit{the $\varphi$-convergence to zero condition}, if
there exists a sequence $\{ t_n \} \subset (0, + \infty)$ such that for any sequence $\{ z_n \} \subset Q$ satisfying
the inequality $\dist(0, \Phi(z_n)) \le t_n$ for all $n \in \mathbb{N}$ one has 
$\limsup_{n \to \infty} \varphi(z_n, \mu_n) \le 0$.
\end{definition}

\begin{corollary} \label{crlr:InexactPDMLimits_Improved}
Let the assumptions of Lemma~\ref{lem:InexactPDMLimits} hold true and suppose that the sequence $\{ \mu_n \}$ satisfies 
the $\varphi$-convergence to zero condition. Then one has $\limsup_{n \to \infty} \varphi(x_n, \mu_n) \le \varepsilon_*$
and the following inequalities hold true:
\begin{align} \label{eq:ObjFuncInexactLim_Improved}
  \vartheta_* \le \liminf_{n \to \infty} f(x_n) &\le \limsup_{n \to \infty} f(x_n) \le \vartheta_* + \varepsilon_*,
  \\ \label{eq:MeritFuncInexactLim_Improved}
  \vartheta_* \le \liminf_{n \to \infty} F(x_n, \mu_n)
  &\le \limsup_{n \to \infty} F(x_n, \mu_n) \le \vartheta_* + \varepsilon_*,
  \\ \label{eq:DualFuncInexactLim_Improved}
  \vartheta_* - \varepsilon_* \le \liminf_{n \to \infty} \Theta(\mu_n)
  &\le \limsup_{n \to \infty} \Theta(\mu_n) \le \vartheta_*.
\end{align}
where $\vartheta_* = \min\{ f_*, \liminf_{y \to 0} \beta(y) \}$.
\end{corollary}

\begin{proof}
Denote $\beta_* = \liminf_{y \to 0} \beta(y)$. If $f_* \le \beta_*$, then the claim of the corollary follows directly 
from Lemma~\ref{lem:InexactPDMLimits}. Therefore, suppose that $\beta_* < f_*$. Note that the lower estimates in
\eqref{eq:ObjFuncInexactLim_Improved}--\eqref{eq:DualFuncInexactLim_Improved} also follow from 
Lemma~\ref{lem:InexactPDMLimits}, which means that we need to prove only the upper estimates.

By the definition of lower limit there exists a sequence $\{ y_k \} \subset Y$ such that $y_k \to 0$ 
and $\beta(y_k) \to \beta_*$ as $k \to \infty$. Since the sequence $\{ y_k \}$ converges to zero, there exists a
subsequence $\{ y_{k_n} \}$ such that $\| y_{k_n} \| \le t_n$ for all $n \in \mathbb{N}$, where $t_n$ is from 
Definition~\ref{def:PhiConvToZero}. In turn, with the use of the definition of the optimal value function $\beta$ 
one can find a sequence $\{ z_n \} \subset Q$ such that 
\[
  0 \in \Phi(z_n) - y_{k_n} \quad \forall n \in \mathbb{N}, \quad 
  \lim_{n \to \infty} f(z_n) = \lim_{n \to \infty} \beta(y_{k_n}).
\]
Note that $\dist(0, \Phi(z_n)) \le \| y_{k_n} \| \le t_n$ for all $n \in \mathbb{N}$. Hence by the 
$\varphi$-convergence to zero condition one has $\limsup_{n \to \infty} \varphi(z_n, \mu_n) \le 0$ and
\begin{align*}
  \limsup_{n \to \infty} \Theta(\mu_n) \le \limsup_{n \to \infty} F(z_n, \mu_n)
  &\le \limsup_{n \to \infty} f(z_n) + \limsup_{n \to \infty} \varphi(z_n, \mu_n)
  \\
  &\le \limsup_{n \to \infty} f(z_n) = \lim_{n \to \infty} \beta(y_{k_n}) = \beta_*,
\end{align*}
which means that the upper estimate in \eqref{eq:DualFuncInexactLim_Improved} holds true.

Recall that $F(x_n, \mu_n) \le \Theta(\mu_n) + \varepsilon_n$ for all $n \in \mathbb{N}$ by the definition of $x_n$.
Consequently, the upper estimate in \eqref{eq:MeritFuncInexactLim_Improved} follows directly from the upper estimate 
in \eqref{eq:DualFuncInexactLim_Improved}. With the use of this upper estimate and the second inequality in 
\eqref{eq:BasicAssumptions} one obtains
\begin{align*}
  \limsup_{n \to \infty} f(x_n) &= \limsup_{n \to \infty} \Big( F(x_n, \mu_n) - \varphi(x_n, \mu_n) \Big)
  \\
  &\le \limsup_{n \to \infty} F(x_n, \mu_n) + \limsup_{n \to \infty} \big( - \varphi(x_n, \mu_n) \big)
  \\
  &\le \vartheta_* + \varepsilon_* - \liminf_{n \to \infty} \varphi(x_n, \mu_n)
  \le \vartheta_* + \varepsilon_*,
\end{align*}
that is, the upper estimate in \eqref{eq:ObjFuncInexactLim_Improved} is satisfied. Finally, note that 
\begin{align*}
  \limsup_{n \to \infty} \varphi(x_n, \mu_n) &= \limsup_{n \to \infty} \Big( F(x_n, \mu_n) - f(x_n) \Big)
  \\
  &\le \limsup_{n \to \infty} F(x_n, \mu_n) - \liminf_{n \to \infty} f(x_n)
  \le \vartheta_* + \varepsilon_* - \vartheta_* = \varepsilon_*,
\end{align*}
which completes the proof of the corollary.
\end{proof}

With the use of Lemma~\ref{lem:InexactPDMLimits} and Corollary~\ref{crlr:InexactPDMLimits_Improved} we can prove 
a general convergence theorem for the inexact primal-dual method (Algorithm~\ref{alg:PrimalDualMethod}), which we call
the universal convergence theorem.

\begin{theorem}[universal convergence theorem for inexact primal-dual methods] \label{thrm:UnivConvThrmInexact}
Let the weak duality hold and $\{ (x_n, \mu_n) \}$ be the sequence generated by the inexact primal-dual method. Suppose
that 
\[
  \lim_{n \to \infty} \dist(0, \Phi(x_n)) = 0, \quad 
  \liminf_{n \to \infty} \varphi(x_n, \mu_n) \ge 0.
\]
Then 
\begin{align} \label{eq:ObjFuncInexactLim_StrDual}
  \underline{\vartheta}_* \le \liminf_{n \to \infty} f(x_n)
  &\le \limsup_{n \to \infty} f(x_n) \le \overline{\vartheta}_* + \varepsilon_*,
  \\ \label{eq:MeritFuncInexactLim_StrDual}
  \underline{\vartheta}_* \le \liminf_{n \to \infty} F(x_n, \mu_n)
  &\le \limsup_{n \to \infty} F(x_n, \mu_n) \le \overline{\vartheta}_* + \varepsilon_*,
  \\ \label{eq:DualFuncInexactLim_StrDual}
  \underline{\vartheta}_* - \varepsilon_* \le \liminf_{n \to \infty} \Theta(\mu_n)
  &\le \limsup_{n \to \infty} \Theta(\mu_n) \le \overline{\vartheta}_*,
\end{align}
where $\underline{\vartheta}_* = \min\{ f_*, \liminf_{y \to 0} \beta(y) \}$,
$\overline{\vartheta}_* = \underline{\vartheta}_*$, if the sequence $\{ \mu_n \}$ satisfies the $\varphi$-convergence to
zero condition, and $\overline{\vartheta}_* = f_*$ otherwise.

In addition, if the optimal value function $\beta$ is lower semicontinuous at the origin, then 
$\underline{\vartheta}_* = \overline{\vartheta}_* = f_*$, 
$\limsup_{n \to \infty} \varphi(x_n, \mu_n) \le \varepsilon_*$, and the two following statements hold true:
\begin{enumerate}
\item{if the functions $f$ and $\dist(0, \Phi(\cdot))$ are lsc on $Q$ and the set-valued map $\Phi$ takes closed values
on the set $Q$, then any limit point of the sequence $\{ x_n \}$ (if such point exists) is an $\varepsilon_*$-optimal
solution of the problem $(\mathcal{P})$;
}

\item{if the function $\Theta$ is upper semicontinuous (usc), then any limit point of the sequence $\{ \mu_n \}$ (if
such point exists) is an $\varepsilon_*$-optimal solution of the dual problem $(\mathcal{D})$.
}
\end{enumerate}
\end{theorem}

\begin{proof}
Inequalities \eqref{eq:ObjFuncInexactLim_StrDual}--\eqref{eq:DualFuncInexactLim_StrDual} follow directly from 
Lemma~\ref{lem:InexactPDMLimits} and Corollary~\ref{crlr:InexactPDMLimits_Improved}.

Suppose that $\beta$ is lsc at the origin. By the definition of optimal value function $\beta(0) = f_*$. Therefore, 
the lower semicontinuity assumption is equivalent to the condition $\min\{ f_*, \lim_{y \to 0} \beta(y) \} = f_*$, 
which yields $\underline{\vartheta}_* = \overline{\vartheta}_* = f_*$.

With the use of inequalities \eqref{eq:MeritFuncInexactLim_StrDual} and \eqref{eq:ObjFuncInexactLim_StrDual} one gets
\begin{align*}
  f_* + \varepsilon_* \ge \limsup_{n \to \infty} F(x_n, \mu_n) 
  &\ge \liminf_{n \to \infty} f(x_n) + \limsup_{n \to \infty} \varphi(x_n, \mu_n)
  \\
  &\ge f_* + \limsup_{n \to \infty} \varphi(x_n, \mu_n),
\end{align*}
which implies that $\limsup_{n \to \infty} \varphi(x_n, \mu_n) \le \varepsilon_*$.

Let $x_*$ be a limit point of the sequence $\{ x_n \}$, that is, some subsequence $\{ x_{n_k} \}$ converges to $x_*$.
Note that $x_* \in Q$, since $Q$ is a closed set. In addition, $\dist(0, \Phi(x_*)) = 0$ due to the facts that 
$\dist(0, \Phi(x_n)) \to 0$ as $n \to \infty$ and the function $\dist(0, \Phi(\cdot))$ is lsc. Hence bearing in mind
the fact that the set $\Phi(x_*)$ is closed by our assumption one can conclude that $x_*$ is a feasible point of the
problem $(\mathcal{P})$. Moreover, by the upper estimate in \eqref{eq:ObjFuncInexactLim_StrDual} and the lower
semicontinuity of $f$ one gets that $f(x_*) \le f_* + \varepsilon_*$, that is, $x_*$ is an $\varepsilon_*$-optimal
solution of the problem $(\mathcal{P})$. The statement about limit points of the sequence $\{ \mu_n \}$ is proved in
the same way.
\end{proof}

For the sake of completeness, let us also point out how the previous theorem can be strengthened in the case of 
the asymptotically exact primal-dual method, that is, in the case when $\varepsilon_n \to 0$ as $n \to \infty$
(or, equivalently, $\varepsilon_* = 0$).

\begin{theorem}[universal convergence theorem for asymptotically exact primal-du\-al methods]
\label{thrm:UnivConvThrmAsympExact}
Let the weak duality hold, the optimal value function $\beta$ be lsc at the origin, and $\{ (x_n, \mu_n) \}$ be 
the sequence generated by the asymptotically exact primal-dual method. Suppose that 
\[
  \lim_{n \to \infty} \dist(0, \Phi(x_n)) = 0, \quad 
  \liminf_{n \to \infty} \varphi(x_n, \mu_n) \ge 0.
\]
Then the strong duality holds, $\lim_{n \to \infty} \varphi(x_n, \mu_n) = 0$, and
\begin{equation} \label{eq:AsymptExactMethod_Limits}
  f_* = \lim_{n \to \infty} f(x_n) = \lim_{n \to \infty} F(x_n, \mu_n) = \lim_{n \to \infty} \Theta(\mu_n) = \Theta_*. 
\end{equation}
In addition, the following statements hold true:
\begin{enumerate}
\item{if the functions $f$ and $\dist(0, \Phi(\cdot))$ are lsc on $Q$ and the set-valued map $\Phi$ takes closed values 
on the set $Q$, then any limit point of the sequence $\{ x_n \}$ (if such point exists) is a globally optimal solution 
of the problem $(\mathcal{P})$;
}

\item{if the function $\Theta$ is usc, then any limit point of the sequence $\{ \mu_n \}$ (if such point exists) is 
a globally optimal solution of the dual problem $(\mathcal{D})$.}
\end{enumerate}
\end{theorem}

\begin{proof}
Note that by the weak duality and the definition of $\Theta_*$ one has
\[
  \limsup_{n \to \infty} \Theta(\mu_n) \le \Theta_* \le f_*.
\]
Bearing in mind these inequalities and putting $\varepsilon_* = 0$ in Theorem~\ref{thrm:UnivConvThrmInexact} one
obtains the required result.
\end{proof}

\begin{remark}
In many particular cases the assumption that the optimal value function $\beta$ is lsc at the origin is, in fact,
equivalent to the strong duality \cite{RubinovHuangYang2002,Dolgopolik2024}. Therefore, in these cases the previous 
theorem can be restated as follows. If the strong duality and conditions \eqref{eq:BasicAssumptions} hold, then 
equalities \eqref{eq:AsymptExactMethod_Limits} are satisfied.
\end{remark}

In the case when either $X$ or $M$ is a reflexive normed space (in particular, finite dimensional normed space) one can 
prove slightly stronger results. For the sake of shortness we formulate these results only for the asymptotically exact
primal-dual method.

\begin{corollary} \label{crlr:UnivConvThrm_RelfexiveCase}
Let $X$ be a reflexive normed space, the functions $f$ and $\dist(0, \Phi(\cdot))$ be weakly sequentially lsc on $Q$,
the set-valued map $\Phi$ take closed values on the set $Q$, the set $Q$ be weakly sequentially closed (in particular,
convex), and the assumptions of Theorem~\ref{thrm:UnivConvThrmAsympExact} be valid. Then the following statements hold 
true:
\begin{enumerate}
\item{any weakly limit point of the sequence $\{ x_n \}$ (if such point exists) is a globally optimal solution of 
the problem $(\mathcal{P})$;}

\item{if the sequence $\{ x_n \}$ is bounded and a globally optimal solution of the problem $(\mathcal{P})$ is unique,
then the sequence $\{ x_n \}$ weakly converges to this globally optimal solution.}
\end{enumerate}
\end{corollary}

\begin{proof}
The validity of the first statement is proved in exactly the same way as the validity of the analogous statement from
Theorem~\ref{thrm:UnivConvThrmInexact}.

Let us prove the second statement. Indeed, suppose by contradiction that the sequence $\{ x_n \}$ does not weakly
converge to unique globally optimal solution $x_*$ of the problem $(\mathcal{P})$. Then there exists a subsequence 
$\{ x_{n_k} \}$ and a neighbourhood $U(x_*)$ of $x_*$ in the weak topology such that $x_{n_k} \notin U(x_*)$ for all 
$k \in \mathbb{N}$.

By our assumption the sequence $\{ x_{n_k} \}$ is bounded. Hence taking into account the fact that the space $X$ is
reflexive one can conclude that there exists a subsequence $\{ x_{n_{k_\ell}} \}$ that weakly converges to some point
$z_*$. By the first statement of the corollary the point $z_*$ is a globally optimal solution of the problem
$(\mathcal{P})$. Thus, $z_* = x_*$, which is impossible, since $\{ x_{n_{k_\ell}} \}$ weakly converges to $z_*$, but
$x_{n_k} \notin U(x_*)$ for all $k \in \mathbb{N}$.
\end{proof}

\begin{corollary}
Let $M$ be a reflexive normed space, the function $\Theta$ weakly sequentially usc, and the assumptions of
Theorem~\ref{thrm:UnivConvThrmAsympExact} be valid. Then the following statements hold true:
\begin{enumerate}
\item{any weakly limit point of the sequence $\{ \mu_n \}$ (if such point exists) is a globally optimal solution
of the dual problem $(\mathcal{D})$;}

\item{if the sequence $\{ \mu_n \}$ is bounded and a globally optimal solution of the dual problem is unique,
then the sequence $\{ \mu_n \}$ weakly converges to this globally optimal solution.}
\end{enumerate}
\end{corollary}

\begin{remark}
All results of this section can be slightly generalized in the following way. Namely, one can easily verify that they 
remain to hold true, if instead of the weak duality one simply assumes that 
$\limsup_{n \to \infty} \Theta(\mu_n) \le f_*$.
\end{remark}

\section{Applications}
\label{sect:Applications}

The goal of this section is to show that many existing convergence theorems for various penalty and augmented Lagrangian
methods can be easily proved with the use of the universal convergence theorem for the inexact primal-dual method
(Theorem~\ref{thrm:UnivConvThrmInexact}) in a unified manner.

\subsection{Primal-dual penalty method}
\label{subsect:PenaltyMethod}

Let us first consider a simple illustrative example of a primal-dual penalty method for the problem $(\mathcal{P})$. 
Let $\mathbb{R}_+ = [0, + \infty)$ and $\omega \colon \mathbb{R}_+ \to \mathbb{R}_+ \cup \{ + \infty \}$ be a
nondecreasing function such that $\dom \omega = [0, \tau)$ for some $\tau \in (0, + \infty]$, $\omega(0) = 0$, 
the function $\omega$ is strictly increasing and continuous on $[0, \tau)$, and $\omega(t) \to + \infty$ as 
$t \to \tau$. Introduce the penalty function for the problem $(\mathcal{P})$ as follows:
\begin{equation} \label{eq:PenaltyFunction}
  F(x, c) = f(x) + c \omega\big( \dist(0, \Phi(x)) \big) \quad x \in X, \: c \ge 0.
\end{equation}
Here $c \ge 0$ is the penalty parameter and $\omega(+ \infty) = + \infty$ by definition. In this case 
$M = [0, + \infty)$, $\mu = c$, and $\varphi(x, \mu) = \mu \omega( \dist(0, \Phi(x)))$. The dual function
\[
  \Theta(c) = \inf_{x \in Q} \Big( f(x) + c \omega\big( \dist(0, \Phi(x)) \big) \Big) \quad \forall c \ge 0.
\]
in this case is obviously concave and usc. 

If a point $x_n \in Q$ is an $\varepsilon_n$-optimal solution of the problem
\begin{equation} \label{eq:PenaltySubproblem}
  \min_x \enspace f(x) + c_n \omega\big( \dist(0, \Phi(x)) \big) \quad \text{subject to} \quad x \in Q
\end{equation}
for some $c_n > 0$ and $\varepsilon_n > 0$, then by the definitions of $\Theta$ and $x_n$ for any $c \ge 0$ one has
\begin{align*}
  \Theta(c) - \Theta(c_n) &\le f(x_n) + c \omega\big( \dist(0, \Phi(x_n)) \big)
  - f(x_n) - c_n \omega\big( \dist(0, \Phi(x_n)) \big) + \varepsilon_n
  \\
  &= \omega\big( \dist(0, \Phi(x_n)) \big) \big( c - c_n \big) + \varepsilon_n.
\end{align*}
Thus, $p = \omega(\dist(0, \Phi(x_n)))$ is an $\varepsilon_n$-supergradient of the concave function $\Theta$, and one
can use $p$ in order to define a subgradient-type step as an update of the penalty parameter:
\[
  c_{n + 1} = c_n + s_n \big[ \omega(\dist(0, \Phi(x_n))) + \delta_n \big], \quad s_n > 0, \quad \delta_n \ge 0.
\] 
Here $s_n$ is the step length, while $\delta_n \ge 0$ is a correction. Based on this penalty parameter update, we can
define the following primal-dual penalty method given in Algorithm~\ref{alg:PrimalDualPenaltyMethod}. Let us note that
$p = \omega(\dist(0, \Phi(x_n)))$ is necessarily finite, since the optimal value of problem
\eqref{eq:PenaltySubproblem} does not exceed $f_*$ for any $c_n \ge 0$.

\begin{algorithm}
\caption{Primal-dual penalty method}

\textbf{Initialization.} Choose an initial value of the penalty parameter $c_0 \ge 0$ and a bounded above sequence of
tolerances $\{ \varepsilon_n \} \subset [0, + \infty)$. Put $n = 0$.

\textbf{Step 1. Solution of subproblem.} Find an $\varepsilon_n$-optimal solution $x_n$ of the problem
\[
  \min_x \: f(x) + c_n \omega\big( \dist(0, \Phi(x)) \big) \quad \text{subject to} \quad x \in Q.
\]

\textbf{Step 2. Penalty parameter update.} Choose some $s_n > 0$, $\delta_n \ge 0$, and define
\[
  c_{n + 1} = c_n + s_n\big[ \omega\big( \dist(0,  \Phi(x_n)) \big) + \delta_n \big].
\]
Increment $n$ and go to \textbf{Step 1}.
\label{alg:PrimalDualPenaltyMethod}
\end{algorithm}

Let us show how one can easily prove convergence of the primal-dual penalty method with the use of the universal
convergence theorem. For the sake of shortness we will only consider the general non-reflexive case and will not show
how Corollary~\ref{crlr:UnivConvThrm_RelfexiveCase} can be reformulated for the primal-dual penalty method.

As above, $\varepsilon_* = \limsup_{n \to \infty} \varepsilon_n$. Recall also that the penalty function $F$ is called
\textit{globally exact} \cite{Demyanov2010,Zaslavski,Dolgopolik2016}, if there exists $c_n > 0$ such that problem
\eqref{eq:PenaltySubproblem} has the same globally optimal solutions as the problem $(\mathcal{P})$. The infimum of all
such $c_n$ is denoted by $c_*(F)$ and is called \textit{the least exact penalty parameter} for the penalty function $F$.

\begin{theorem}
Let the penalty function $F(\cdot, c_0)$ be bounded below and the sequence $\{ (x_n, c_n) \}$ be generated by the
primal-dual penalty method. Suppose also that there exists $s_* > 0$ such that $s_n \ge s_*$ for all $n \in \mathbb{N}$.
Then $\dist(0, \Phi(x_n)) \to 0$ as $n \to \infty$ and
\begin{align*}
  \min\{ f_*, \liminf_{y \to 0} \beta(y) \} &\le \liminf_{n \to \infty} f(x_n) 
  \\
  &\le \limsup_{n \to \infty} f(x_n) \le \min\{ f_*, \liminf_{y \to 0} \beta(y) \} + \varepsilon_*.
\end{align*}
In addition, if the optimal value function $\beta$ is lsc at the origin, then
\[
  \limsup_{n \to \infty} c_n \omega\big( \dist(0, \Phi(x_n)) \big) \le \varepsilon_*
\]
and the following statements hold true:
\begin{enumerate}
\item{if the functions $f$ and $\dist(0, \Phi(\cdot))$ are lsc on $Q$ and the set-valued map $\Phi$ takes closed values
on the set $Q$, then any limit point of the sequence $\{ x_n \}$ (if such point exists) is an $\varepsilon_*$-optimal
solution of the problem $(\mathcal{P})$;}

\item{if $\varepsilon_* = 0$ and the sequence $\{ c_n \}$ is bounded, then the penalty function $F$ is globally exact
and $c_* := \lim_{n \to \infty} c_n \ge c_*(F)$.}
\end{enumerate}
\end{theorem}

\begin{proof}
Since the sequence $\{ c_n \}$ is obviously nondecreasing, the boundedness below of the function $F(\cdot, c_0)$
guarantees that Assumption~\ref{as:PrimalSequenceWellDefined} is satisfied.

As was noted above, in the case under consideration one can define $\varphi(x, \mu) = \mu \omega(\dist(0, \Phi(x))$ for
any $\mu \in M = \mathbb{R}_+$ and $x \in X$. This function is nonnegative and $\varphi(x, \mu) = 0$ for any 
$\mu \in M$ and any feasible point $x$. Hence by Lemma~\ref{lem:WeakDuality} the weak duality holds and
$\liminf_{n \to \infty} \varphi(x_n, c_n) \ge 0$. Furthermore, the sequence $\{ c_n \}$ obviously satisfies the
$\varphi$-convergence to zero condition with any sequence $\{ t_n \} \subset (0, + \infty)$ such that 
$c_n \omega(t_n) \to 0$ as $n \to \infty$ (at least one such sequence $\{ t_n \}$ exists, since $\omega(0) = 0$ and 
$\omega$ is continuous on some interval $[0, \tau)$). Therefore, in order to apply the universal convergence theorem 
for the inexact primal-dual method it is sufficient to check that $\dist(0, \Phi(x_n)) \to 0$ as $n \to \infty$. 

Suppose by contradiction that $\dist(0, \Phi(x_n))$ does not converge to zero. Then the sequence 
$\omega(\dist(0, \Phi(x_n))$ does not converge to zero either and, therefore,
\begin{equation} \label{eq:DiverPenDistSeries}
  \sum_{n = 0}^{\infty} \omega\big( \dist(0, \Phi(x_n)) \big) = + \infty.
\end{equation}
According to Step~2 of Algorithm~\ref{alg:PrimalDualPenaltyMethod} and our assumption on the sequence $\{ s_n \}$ one
has
\[
  c_n = c_0 + \sum_{i = 0}^{n - 1} \Big( s_i \omega\big( \dist(0, \Phi(x_i)) \big) + \delta_i \Big)
  \ge c_0 + s_* \sum_{i = 0}^{n - 1} \omega\big( \dist(0, \Phi(x_i)) \big)
\]
for any $n \in \mathbb{N}$. Hence taking into account \eqref{eq:DiverPenDistSeries} one can conclude that 
$c_n \to \infty$ as $n \to \infty$.

On the other hand, note that
\begin{align*}
  F(x_n, c_n) &= F(x_n, c_0) + (c_n - c_0) \omega\big( \dist(0, \Phi(x_n) \big)
  \\
  &\ge \inf_{x \in Q} F(x, c_0) + (c_n - c_0) \omega\big( \dist(0, \Phi(x_n) \big)
\end{align*}
for any $n \in \mathbb{N}$. Since the sequence $\{ c_n \}$ increases unboundedly, while the sequence 
$\{ \omega(\dist(0, \Phi(x_n)) \}$ does not converge to zero, the inequality above implies that
\begin{equation} \label{eq:PenFuncValuesUnbounded}
  \limsup_{n \to \infty} F(x_n, c_n) = + \infty.
\end{equation}
However, by the weak duality one has
\[
  F(x_n, c_n) \le \inf_{x \in Q} F(x, c_n) + \varepsilon_n \le f_* + \varepsilon_n,
\]
which contradicts \eqref{eq:PenFuncValuesUnbounded}. Thus, $\dist(0, \Phi(x_n)) \to 0$ as $n \to \infty$ and all claims
of the theorem, except for the claim about the sequence $\{ c_n \}$, hold true by
Theorem~\ref{thrm:UnivConvThrmInexact}.

If $\varepsilon_* = 0$ and the sequence $\{ c_n \}$ is bounded, then its limit $c_*$ (which exists since the sequence 
$\{ c_n \}$ is nondecreasing) is a globally optimal solution of the dual problem and the strong duality holds by
Theorem~\ref{thrm:UnivConvThrmAsympExact}. Thus, $\Theta(c_*) = f_*$, which by 
\cite[Corollary~3.3 and Remark~8]{Dolgopolik2016} means that the penalty function $F$ is globally exact and 
$c_* \ge c_*(F)$.
\end{proof}

\begin{remark}
{(i)~One can readily verify that in the case of penalty function \eqref{eq:PenaltyFunction} the lower semicontinuity of
the perturbation function $\beta$ is equivalent to the strong duality (see, e.g. \cite[Theorem~3.16]{Dolgopolik2017}).
}

{(ii)~By the theorem above the global exactness of the penalty function $F$ is a necessary condition for the sequence
of penalty parameters $\{ c_n \}$ generated by the primal-dual penalty method to be bounded (or, equivalently, to
converge). One can show that under some additional assumptions (in particular, the assumption that the sequence 
$\{ \delta_n \}$ has a subsequence that is bounded away from zero) the global exactness of the penalty function $F$ is
also sufficient for the boundedness of the sequence of penalty parameters. See 
\cite[Theorem~5 and Remark~7]{Dolgopolik2022} for more details.
}
\end{remark}

\subsection{Primal-dual rounded weighted $\ell_1$-penalty method}
\label{subsect:RoundedWeightedPenaltyMethod}

Let us now consider the primal-dual rounded weighted $\ell_1$-penalty method from \cite{BurachikKayaPrice} that can be
viewed as a primal-dual method for the $\ell_1$-penalty function with multidimensional penalty parameter (i.e.
individual penalty parameter for each constraint; see \cite{Dolgopolik2022}) combined with smoothing techniques 
\cite{Pinar,YangMengHuang,MengDangYang,LiuzziLucidi,Grossmann}.

In this section we study a primal-dual method for the following mathematical programming problem:
\begin{equation} \label{prob:MathProgram}
  \min \: f(x) \quad \text{subject to} \quad h(x) = 0, \quad g(x) \le 0, \quad x \in Q,
\end{equation}
where $h \colon X \to \mathbb{R}^m$, $h = (h_1, \ldots, h_m)$, and $g \colon X \to \mathbb{R}^{\ell}$, 
$g = (g_1, \ldots, g_{\ell})$, are given functions, and the inequality $g(x) \le 0$ is understood coordinate-wise.
This problem can be rewritten as the problem $(\mathcal{P})$, if one defines $Y = \mathbb{R}^{m + \ell}$ and
$\Phi(\cdot) := \{ h(\cdot) \} \times \{ g(\cdot) + \mathbb{R}_+^{\ell} \}$. If the space $Y$ is equipped with 
the $\ell_1$-norm, then the infeasibility measure takes the following form:
\begin{equation} \label{eq:Ell1_InfeasMeas}
  \dist(0, \Phi(x)) = \sum_{i = 1}^m |h_i(x)| + \sum_{j = 1}^{\ell} \max\{ 0, g_j(x) \}, \quad x \in X.
\end{equation}
The classic $\ell_1$-penalty function for problem \eqref{prob:MathProgram} is defined as follows:
\[
  F_{\ell_1}(x, c) = f(x) + c \Big( \sum_{i = 1}^m |h_i(x)| + \sum_{j = 1}^{\ell} \max\{ 0, g_j(x) \} \Big),
\]
where $c > 0$ is the penalty parameter.

Burachik et al. in \cite{BurachikKayaPrice} proposed to consider a rounded weighted $\ell_1$-penalty function for
problem \eqref{prob:MathProgram}, that is, they proposed to replace the nonsmooth functions $|h_i(\cdot)|$ and 
$\max\{ 0, g_j(\cdot) \}$ with their smoothing approximations (see 
\cite{Pinar,YangMengHuang,MengDangYang,LiuzziLucidi,Grossmann} for more details on smoothing approximations of 
nonsmooth penalty functions) and to use individual penalty parameter for each of the constraint $h_i(x) = 0$ and 
$g_j(x) \le 0$ instead of the single penalty parameter $c$. To define this penalty function, first for any $w > 0$ 
we introduce the following smoothing approximations of the functions $|\cdot|$ and $\max\{ 0, \cdot \}$:
\[
  \eta(t, w) = \begin{cases}
    \frac{t^2}{2w}, & \text{if } |t| < w,
    \\
    |t| - \frac{w}{2}, & \text{if } |t| \ge w,
  \end{cases}
  \qquad
  \gamma(t, w) = \begin{cases}
    \frac{t^2}{2w}, & \text{if } 0 < t < w,
    \\
    t - \frac{w}{2}, & \text{if } t \ge w,
    \\
    0, & \text{if } t \le 0.
  \end{cases}
\]
By definition we set $\eta(t, 0) = |t|$ and $\gamma(t, 0) = \max\{ 0, t \}$ for all $t \in \mathbb{R}$.

Introduce the rounded weighted $\ell_1$-penalty function for problem \eqref{prob:MathProgram} as follows:
\[
  F\big(x, (u, v, w) \big) := f(x) + \sum_{i = 1}^m u_i \eta(h_i(x), w)
  + \sum_{j = 1}^{\ell} v_j \gamma(g_j(x), w)
\]
Here $u \in \mathbb{R}^m_+$ and $v \in \mathbb{R}^{\ell}_+$ are \textit{multidimensional} penalty parameters and 
$w \ge 0$ is \textit{the smoothing parameter}. In this case one can define $M = \mathbb{R}^{m + \ell + 1}$,
\begin{equation} \label{eq:RoundedWeightedPenTerm}
  \varphi(x, \mu) = \sum_{i = 1}^m u_i \eta(h_i(x), w) + \sum_{j = 1}^{\ell} v_j \gamma(g_j(x), w),
\end{equation}
for any $x \in X$ and $\mu = (u, v, w) \in M$, and the dual function 
\[
  \Theta(\mu) = \inf_{x \in Q} \Big( f(x) + \sum_{i = 1}^m u_i \eta(h_i(x), w) 
  + \sum_{j = 1}^{\ell} v_j \gamma(g_j(x), w) \Big).
\]
The function $\Theta$ is obviously usc as the infimum of a family of continuous functions, while the function 
$(u, v) \mapsto \Theta(u, v, w)$ is concave as the infimum of a family of affine functions. Arguing in the same way as 
in the previous section one can verify that if $x_n$ is an $\varepsilon_n$-optimal solution of the problem
\begin{equation} \label{prob:RoundedWeightedPenSubprob}
  \min_x \enspace F(x, \mu_n) \quad \text{subject to} \quad x \in Q
\end{equation}
for some $\varepsilon_n \ge 0$ and $\mu_n = (u_n, v_n, w_n) \in M$, then the vector $p_n \in \mathbb{R}^{m + \ell}$,
\[
  p_n = \Big( \eta(h_1(x_n), w_n), \ldots, \eta(h_m(x_n), w_n), 
  \gamma(g_1(x_n), w_n), \ldots, \gamma(g_{\ell}(x_n), w_n) \Big)
\]
is an $\varepsilon_n$-supergradient of the concave function $(u, v) \mapsto \Theta(u, v, w_n)$ (note that each 
coordinate of the vector $p_n$ is finite, since the optimal value of problem \eqref{prob:RoundedWeightedPenSubprob}
does not exceed $f_*$). Therefore one can use the following subgradient-type step as an update of the multidimensional
penalty parameter:
\[
  (u_{n + 1}, v_{n + 1}) = (u_n, v_n) + s_n \big( p_n + \delta_n \big),  
\]
where $s_n > 0$ is the step size, while $\delta_n \ge 0$ is a correction. Utilising this penalty parameter update and
taking into account the fact that $\eta(\cdot, w) \to |\cdot|$ and $\gamma(\cdot, w) \to \max\{ 0, \cdot \}$ as 
$w \to 0$ uniformly on $\mathbb{R}$, one can define the following primal-dual rounded weighted $\ell_1$-penalty method 
given in Algorithm~\ref{alg:PD_RoundedWeightedPenaltyMethod} that was first introduced by Burachik et al. in
\cite{BurachikKayaPrice} in the case when $\varepsilon_n \equiv 0$ and the set $Q$ is compact.

\begin{algorithm}
\caption{Primal-dual rounded weighted-$\ell_1$ penalty method}

\textbf{Initialization.} Choose an initial value of the multidimensional penalty parameters $u_0 \ge 0$ and $v_0 \ge 0$,
an initial value of the smoothing parameter $w_0 > 0$, and a bounded above sequence of tolerances 
$\{ \varepsilon_n \} \subset [0, + \infty)$. Put $n = 0$.

\textbf{Step 1. Solution of subproblem.} Find an $\varepsilon_n$-optimal solution $x_n$ of the problem
\[
  \min_x \: f(x) + \sum_{i = 1}^m u_{ni} \eta(h_i(x), w_n)
  + \sum_{j = 1}^{\ell} v_{nj} \gamma(g_j(x), w_n) \quad \text{s.t.} \quad x \in Q,
\]
where $u_n = (u_{n1}, \ldots, u_{nm})$ and $v_n = (v_{n1}, \ldots, v_{n \ell})$.

\textbf{Step 2. Update of parameters.} Compute the vector
\[
  p_n = \Big( \eta(h_1(x_n), w_n), \ldots, \eta(h_m(x_n), w_n), 
  \gamma(g_1(x_n), w_n), \ldots, \gamma(g_{\ell}(x_n), w_n) \Big),
\]
choose some $s_n > 0$, $\delta_n \in \mathbb{R}^{m + \ell}_+$, and define
\[
  (u_{n + 1}, v_{n + 1}) = (u_n, v_n) + s_n \big( p_n + \delta_n \big).
\]
Choose some $w_{n + 1} \le w_n$, increment $n$, and go to \textbf{Step 1}.
\label{alg:PD_RoundedWeightedPenaltyMethod}
\end{algorithm}

Let us show how one can easily prove convergence of the primal-dual rounded weighted $\ell_1$-penalty method with 
the use of the universal convergence theorem for inexact primal-dual methods. For the sake of convenience, denote
$\mu_n = (u_n, v_n, w_n)$.

\begin{theorem} \label{thrm:PD_RoundedWeightedPenMethod}
Let the penalty function $F(\cdot, \mu_0)$ be bounded below and the sequence $\{ (x_n, u_n, v_n, w_n) \}$ be generated
by the primal-dual rounded weighted $\ell_1$-penalty method. Suppose also that one of the following two conditions is 
valid:
\begin{enumerate} 
\item{there exists $s_* > 0$ such that $s_n \ge s_*$ for all $n \in \mathbb{N}$;}

\item{there exists a nondecreasing function $\omega \colon \mathbb{R}_+ \to \mathbb{R}_+$ such that $\omega(t) > 0$ 
for any $t > 0$, and $s_n \ge \omega(\| p_n \|)$ for all $n \in \mathbb{N}$, where $\| \cdot \|$ is the Euclidean norm.
}
\end{enumerate}
Then $\dist(0, \Phi(x_n)) \to 0$ as $n \to \infty$ and
\begin{align*}
  \vartheta_* \le \liminf_{n \to \infty} f(x_n) &\le \limsup_{n \to \infty} f(x_n) \le \vartheta_* + \varepsilon_*,
  \\
  \vartheta_* - \varepsilon_* \le \liminf_{n \to \infty} \Theta(\mu_n) 
  &\le \limsup_{n \to \infty} \Theta(\mu_n) \le \vartheta_*,
\end{align*}
where $\vartheta_* = \min\{ f_*, \liminf_{y \to 0} \beta(y) \}$. In addition, if the optimal value function $\beta$ be 
lsc at the origin, then
\begin{equation}
  \limsup_{n \to \infty} \Big(\sum_{i = 1}^m u_{ni} \eta(h_i(x_n), w_n)
  + \sum_{j = 1}^{\ell} v_{nj} \gamma(g_j(x_n), w_n) \Big) \le \varepsilon_*
\end{equation}
and the following statements hold true:
\begin{enumerate}
\item{if the function $h$ is continuous on $Q$, and the functions $f$ and $g_j$, $j \in \{ 1, \ldots, \ell \}$, are lsc 
on $Q$, then any limit point of the sequence $\{ x_n \}$ (if such point exists) is an $\varepsilon_*$-optimal solution 
of problem \eqref{prob:MathProgram};
}

\item{if the sequence $\{ (u_n, v_n) \}$ is bounded, then the sequence $\{ (u_n, v_n, w_n) \}$ converges and its limit
point is an $\varepsilon_*$-optimal solutions of the dual problem;
}

\item{in the case $\varepsilon_* = 0$ for the boundedness of the sequence $\{ (u_n, v_n) \}$ it is necessary that
the $\ell_1$-penalty function for problem \eqref{prob:MathProgram} is globally exact.
}
\end{enumerate}
\end{theorem}

\begin{proof}
Note that $u_n \le u_{n + 1}$ and $v_n \le v_{n + 1}$ for any $n \in \mathbb{N}$, since the functions $\eta$ and
$\gamma$ are nonnegative. Furthermore, the functions $\eta(t, w)$ and $\gamma(t, w)$ are nonincreasing in $w$ 
(see \cite{BurachikKayaPrice}). Hence taking into account the way in which the multidimensional penalty parameters
$u_n$ and $v_n$ and the smoothing parameter $w_n$ are updated by the method, one can conclude that
$F(\cdot, \mu_n) \le F(\cdot, \mu_{n + 1})$ for any $n \in \mathbb{N}$. Therefore, the boundedness below of the function
$F(\cdot, \mu_0)$ guarantees that Assumption~\ref{as:PrimalSequenceWellDefined} holds true.

The function $\varphi$ defined in \eqref{eq:RoundedWeightedPenTerm} is obviously nonnegative. Moreover, for any 
$\mu \in M$ and any feasible point $x$ one has $\varphi(x, \mu) = 0$ . Consequently, by Lemma~\ref{lem:WeakDuality} 
the weak duality holds and $\liminf_{n \to \infty} \varphi(x_n, \mu_n) \ge 0$. Moreover, the sequence $\{ \mu_n \}$
clearly satisfies the $\varphi$-convergence to zero condition with any sequence $\{ t_n \} \subset (0, + \infty)$
such that 
\[
  \lim_{n \to \infty} \Big( \sum_{i = 1}^m u_{ni} \eta(t_n, w_n) 
  + \sum_{j = 1}^{\ell} v_{nj} \gamma(t_n, w_n) \Big) = 0
\]
(at least one such sequence $\{ t_n \}$ obviously exists). Thus, in order to apply the universal convergence theorem for
the inexact primal-dual method it remains to check that $\dist(0, \Phi(x_n)) \to 0$ as $n \to \infty$.

Suppose by contradiction that the sequence $\{ \dist(0, \Phi(x_n)) \}$ does not converge to zero. Then taking into
account \eqref{eq:Ell1_InfeasMeas} one can conclude that either the sequence $\{ |h_r(x_n)| \}$ does not converge to
zero for some $r \in \{ 1, \ldots, m \}$ or the sequence $\{ \max\{ 0, g_r(x_n) \} \}$ does not converge to zero for
some $r \in \{ 1, \ldots, \ell \}$. We will consider only the first case, since the proof in the second case almost
literally repeats the proof in the first one.

Thus, we assume that the sequence $\{ |h_r(x_n)| \}$ does not converge to zero for some $r \in \{ 1, \ldots, m \}$. 
Then the sequence $\{ \eta(h_r(x_n), w_0) \}$ does not converge to zero either and
\begin{equation} \label{eq:RoundedAbsInfMeas_DivergSeries}
  \sum_{n = 0}^{\infty} \eta(h_r(x_n), w_0) = + \infty.
\end{equation}
According to Step~2 of Algorithm~\ref{alg:PD_RoundedWeightedPenaltyMethod} one has
\[
  u_{nr} = u_{0r} + \sum_{k = 0}^{n - 1} s_k \big( \eta(h_r(x_k), w_k) + \delta_{kr} \big)
  \ge u_{0r} + \sum_{k = 0}^{n - 1} s_k \eta(h_r(x_k), w_0)
\]
(here we used the facts that $\delta_n \ge 0$, $w_{n + 1} \le w_n$ for all $n$, and the function $\eta(t, w)$ is
nonincreasing in $w$). If $s_n \ge s_* > 0$ for all $n \in \mathbb{N}$, then $u_{nr} \to + \infty$ as $n \to + \infty$
due to \eqref{eq:RoundedAbsInfMeas_DivergSeries}. In turn, if $s_n \ge \omega(\| p_n \|)$ for all $n \in \mathbb{N}$,
then $s_n \ge \omega(\eta(h_r(x_n), w_0))$ due to the fact that the function $\omega$ is nondecreasing. Clearly, the
sequence $\{  \omega(\eta(h_r(x_n), w_0)) \eta(h_r(x_k), w_0) \}$ does not converge to zero, since the sequence 
$\{ \eta(h_r(x_k), w_0) \}$ does not converge to zero, the function $\omega$ is nondecreasing, and $\omega(t) > 0$ for
any $t > 0$. Therefore,
\begin{align*}
  \lim_{n \to \infty} u_{nr} 
  &\ge u_{0r} + \lim_{n \to \infty} \sum_{k = 0}^{n - 1}  \omega(\eta(h_r(x_n), w_0)) \eta(h_r(x_k), w_0)
  \\
  &= u_{0r} + \sum_{k = 0}^{\infty} \omega(\eta(h_r(x_n), w_0)) \eta(h_r(x_k), w_0) = + \infty. 
\end{align*}
Thus, in both cases $u_{nr} \to + \infty$ as $n \to \infty$.

Bearing in mind the facts that $u_n \ge u_0$, $v_n \ge v_0$ and $w_n \le w_0$, and the functions $\eta(t, w)$ and
$\gamma(t, w)$ are nonincreasing in $w$ one obtains:
\begin{align*}
  F(x_n, \mu_n) &= F(x_n, \mu_0) 
  \begin{aligned}[t]
    &+ \sum_{i = 1}^m \big( u_{ni} \eta(h_i(x_n), w_n) - u_{0i} \eta(h_i(x_n), w_0) \big)
    \\
    &+ \sum_{j = 1}^{\ell} \big( v_{nj} \gamma(g_j(x_n), w_n) - v_{0j} \eta(g_j(x_n), w_0) \big)
  \end{aligned}
  \\
  &\ge \inf_{x \in Q} F(x, \mu_0) + \big( u_{nr} - u_{0r} \big) \eta(h_r(x_n), w_0).
\end{align*}
Therefore
\begin{equation} \label{eq:RoundedWeightedEll1PenF_UnboundIncr}
  \limsup_{n \to \infty} F(x_n, \mu_n) = + \infty
\end{equation}
due to the facts that $u_{nr} \to + \infty$ as $n \to \infty$ and the sequence $\{ \eta(h_r(x_n), w_0) \}$ does not
converge to zero. However, by the definition of $x_n$ and the fact that $F(x, \mu_n) = f(x)$ for any feasible $x$ one
has $F(x_n, \mu_n) \le f_* + \varepsilon_n$ for all $n \in \mathbb{N}$, which contradicts
\eqref{eq:RoundedWeightedEll1PenF_UnboundIncr}.

Thus, $\dist(0, \Phi(x_n)) \to 0$ as $n \to \infty$ and all claims of the theorem, except for the claim about the case 
$\varepsilon_* = 0$, hold true by Theorem~\ref{thrm:UnivConvThrmInexact} (one only needs to note that the convergence
of the sequence $\{ (u_n, v_n, w_n) \}$ is equivalent to the boundedness of the sequence $\{ (u_n, v_n) \}$ due to 
the facts that $0 \le w_{n + 1} \le w_n$, $u_n \le u_{n + 1}$, and $v_n \le v_{n + 1}$ for all $n \in \mathbb{N}$).

Let us finally show that in the case $\varepsilon_* = 0$ for the boundedness of the sequence $\{ (u_n, v_n) \}$ it is
necessary that the $\ell_1$-penalty function for problem \eqref{prob:MathProgram} is globally exact. Indeed, if this
sequence is bounded, then the sequence $\{ (u_n, v_n, w_n) \}$ converges to some point $\mu_* = (u_*, v_*, w_*)$. 
By Theorem~\ref{thrm:UnivConvThrmAsympExact} this point is a globally optimal solution of the dual problem and the
strong duality holds true, which implies that $\Theta(\mu_*) = f_*$.

Let $c_* = \max\{ (u_*)_1, \ldots, (u_*)_m, (v_*)_1, \ldots, (v_*)_{\ell} \}$. Then 
\[
  F_{\ell_1}(x, c_*) \ge F(x, \mu_*) \ge \Theta(\mu_*) = f_* \quad \forall x \in Q
\]
(the first inequality follows from the properties of the functions $\eta$ and $\gamma$; see \cite{BurachikKayaPrice}). 
By \cite[Corollary~3.3 and Remark~8]{Dolgopolik2016} the inequality above implies that the penalty function $F_{\ell_1}$
is globally exact. Thus, the boundedness of the sequence $\{ (u_n, v_n) \}$ implies that the $\ell_1$-penalty function
for problem \eqref{prob:MathProgram} is globally exact.
\end{proof}

Let us also show how the previous theorem can be strengthened in the case, when $Q$ is a compact set and the method is 
asymptotically exact, that was considered in \cite{BurachikKayaPrice}.

\begin{corollary}
Let the set $Q$ be compact, the function $h$ be continuous on $Q$, the functions $f$ and $g_j$, 
$j \in \{ 1, \ldots, \ell \}$, be lsc on $Q$, and the sequence $\{ (x_n, u_n, v_n, w_n) \}$ be generated by 
the primal-dual rounded weighted $\ell_1$-penalty method with $\varepsilon_n \to 0$ as $n \to \infty$. Suppose also that
one of the following two conditions is valid:
\begin{enumerate} 
\item{there exists $s_* > 0$ such that $s_n \ge s_*$ for all $n \in \mathbb{N}$;}

\item{there exists a nondecreasing function $\omega \colon \mathbb{R}_+ \to \mathbb{R}_+$ such that $\omega(t) > 0$ for
any $t > 0$, and $s_n \ge \omega(\| p_n \|)$ for all $n \in \mathbb{N}$, where $\| \cdot \|$ is the Euclidean norm.
}
\end{enumerate}
Then
\[
  \lim_{n \to \infty} \dist(0, \Phi(x_n)) = 0, \quad
  \lim_{n \to \infty} f(x_n) = \lim_{n \to \infty} \Theta(\mu_n) = f_*,
\]
the sequence $\{ x_n \}$ is bounded, and all its limit points are globally optimal solutions of the problem
$(\mathcal{P})$. If, in addition, the sequence $\{ (u_n, v_n) \}$ is bounded, then the sequence $\{ (u_n, v_n, w_n) \}$
converges to a globally optimal solution of the dual problem.
\end{corollary}

\begin{proof}
Under the assumptions of the corollary the function $F(\cdot, \mu_0)$ is obviously bounded below on $Q$. Let us check
that the optimal value function $\beta$ is lsc at the origin. Then applying 
Theorem~\ref{thrm:PD_RoundedWeightedPenMethod} with $\varepsilon_* = 0$ we arrive at the desired result.

Let $\{ y_n \} \subset Y := \mathbb{R}^{m + \ell}$ be any sequence such that $y_n \to 0$ as $n \to \infty$ and
\[
  \beta_* := \lim_{n \to \infty} \beta(y_n) = \liminf_{y \to 0} \beta(y).
\]
Note that the sequence $\{ \beta(y_n) \}$ is bounded below due to the compactness of the set $Q$. If there does not
exist $n_0 \in \mathbb{N}$ such that the sequence $\{ \beta(y_n) \}_{n \ge n_0}$ is bounded above, then 
$\beta_* = + \infty$ and the function $\beta$ is obviously lsc at the origin. Therefore, without loss of generality one
can suppose that the sequence $\{ \beta(y_n) \}$ is bounded.

By the definition of $\beta$ for any $n \in \mathbb{N}$ one can find $x_n \in Q$ such that
\begin{align*}
  \beta(y_n) \le f(x_n) + \varepsilon_n, \quad &h_i(x_n) = y_{ni}, \quad i \in \{ 1, \ldots, m \},
  \\
  &g_j(x_n) \le y_{n(m + j)}, \quad j \in \{ 1, \ldots, \ell \}.
\end{align*}
Since $Q$ is compact, one can find a subsequence $\{ x_{n_k} \}$ that converges to some $x_* \in Q$. From 
the inequalities above, the assumptions on the function $h$ and $g$, and the fact that $y_n \to 0$ as $n \to \infty$ it
follows that $h(x_*) = 0$ and $g(x_*) \le 0$, that is, the point $x_*$ is feasible for problem \eqref{prob:MathProgram}.
Moreover, due to the lower semicontinuity of $f$ one has
\[
  \liminf_{y \to 0} \beta(y) = \lim_{n \to \infty} \beta(y_n) = \lim_{n \to \infty} f(x_n) 
  \ge f(x_*) \ge f_* =: \beta(0),
\]
which means that the optimal value function $\beta$ is lsc at the origin.
\end{proof}

\begin{remark}
Let us note that \cite[Theorem~4.1]{BurachikKayaPrice} is a particular case of the corollary above. Furthermore, this 
corollary and Theorem~\ref{thrm:PD_RoundedWeightedPenMethod} significantly generalize the convergence analysis presented
in \cite{BurachikKayaPrice}, since we do not assume that the set $Q$ is compact, $\varepsilon_n \equiv 0$, and
either $s_n \equiv 1 / \| p_n \|^2$ or $s_n \equiv \| p_n \|^2$ (in the first case $s_n \ge s_* > 0$ for all 
$n \in \mathbb{N}$, while in the second case $s_n \ge \omega(\| p_n\|)$ for $\omega(t) = t^2$). Thus, with the use of
the universal convergence theorem for inexact primal-dual methods we not only managed to recover existing results for
the primal-dual rounded weighted $\ell_1$-penalty method, but also significantly generalized these results.
\end{remark}

\subsection{Augmented Lagrangian method for cone constrained optimization}
\label{subsect:AugmLagrMethod}

In this section we study the classic augmented Lagrangian method based on the Hestenes-Powell-Rockafellar augmented
Lagrangian \cite{BirginMartinez}. This method has been extended, in particular, to the case of optimization problems 
with nonlinear second order cone constraints \cite{LiuZhang2007} and nonlinear semidefinite programming problems 
\cite{LuoWuChen2012}. Our aim is to extend this method to the case of general cone constrained optimization problems 
(which to the best of the author's knowledge has never been done before) and prove its global convergence with the use
of the universal convergence theorem for inexact primal-dual methods.

Let $Y$ be a Hilbert space with inner product $\langle \cdot, \cdot \rangle$, $K \subset Y$ be a closed convex cone, 
and $G \colon X \to Y$ be a given function. In this section we consider the following cone constrained optimization 
problem:
\begin{equation} \label{prob:ConeConstrained}
  \min \: f(x) \quad \text{subject to} \quad G(x) \in K, \quad x \in Q.
\end{equation}
This problem can be rewritten as the problem $(\mathcal{P})$ with $\Phi(\cdot) := G(\cdot) - K$. Note that 
the infeasibility measure satisfies the equality $\dist(0, \Phi(x)) = \dist(G(x), K)$.

Denote by $K^* = \{ y^* \in Y \mid \langle y^*, y \rangle \le 0 \: \forall y \in K \}$ \textit{the polar cone} 
of the cone $K$. Since $Y$ is a Hilbert space, for any closed convex set $V \subseteq Y$ and any $y \in Y$ there 
exists a unique metric projection of $y$ onto $V$, which we denote by $\Pi_V(y)$. Following 
\cite[Section~11.K]{RockafellarWets} and \cite{ShapiroSun} one can define an augmented Lagrangian for problem 
\eqref{prob:ConeConstrained} as follows:
\begin{equation} \label{eq:HPRAugmLagr}
\begin{split}
  \mathscr{L}(x, \lambda, c) 
  &= f(x) + \frac{1}{2c} \Big[ \big\| \Pi_{K^*}(\lambda + c G(x)) \big\|^2 - \| \lambda \|^2 \Big]
  \\
  &= f(x) + \frac{1}{2c} \Big[ \dist\big( \lambda + c G(x), K \big)^2 - \| \lambda \|^2 \Big],
\end{split}
\end{equation}
It can be viewed an extension of the Hestenes-Powell-Rockafellar augmented Lagrangian
\cite{Rockafellar73,BirginMartinez} to the case of general cone constrained optimization problems. Here $\lambda \in Y$
is a multiplier and $c > 0$ is a penalty parameter. Note that the dual function
\[
  \Theta(\lambda, c) = \inf_{x \in Q} \mathscr{L}(x, \lambda, c)
\]
is concave in $(\lambda, c)$ and usc (see \cite[Section~11.K]{RockafellarWets} and \cite{ShapiroSun}).

Let us now describe a basic augmented Lagrangian method. Its theoretical scheme is given in 
Algorithm~\ref{alg:BasicAugmLagrMethod}. 

\begin{algorithm}
\caption{Inexact basic augmented Lagrangian method}

\textbf{Initialization.} Choose initial values of the multiplier $\lambda_0 \in K^*$ and the penalty parameter 
$c_0 >0$, and a bounded above sequence of tolerances $\{ \varepsilon_n \} \subset [0, + \infty)$. Put $n = 0$.

\textbf{Step 1. Solution of subproblem.} Find an $\varepsilon_n$-optimal solution $x_n$ of the problem
\[
  \min_x \: \mathscr{L}(x, \lambda_n, c_n) \quad \text{subject to} \quad x \in Q.
\]

\textbf{Step 2. Multiplier update.} Choose some $\lambda_{n + 1} \in K^*$.

\textbf{Step~3. Penalty parameter update.} Choose some $c_{n + 1} \ge c_n$, increment $n$, and go to
\textbf{Step 1}.
\label{alg:BasicAugmLagrMethod}
\end{algorithm}

In order to incorporate several particular cases into one method, we do not specify the way in which the multipliers 
and the penalty parameter are updated in this method. In particular, one can define
\[
  \lambda_{n + 1} = \lambda_n + \Big( \Pi_{K^*}(\lambda_n + c_n G(x_n)) - \lambda_n \Big)
\]
as in \cite[Algorithm~4.1]{BirginMartinez}, \cite{LiuZhang2007}, \cite[Algorithm~1]{LuoSunWu}, 
\cite[Algorithm~1]{LuoWuChen2012}, or one can put
\begin{equation} \label{eq:MultUpdated_Safeguarding}
  \lambda_{n + 1} = \Pi_{T \cap K^*}
  \left[ \lambda_n + \Big( \Pi_{K^*}(\lambda_n + c_n G(x_n)) - \lambda_n \Big) \right]
\end{equation}
for some convex bounded set $T \subset Y$ containing a neighbourhood of zero, as in the multiplier method with 
safeguarding (see \cite[Algorithm~2]{LuoSunWu} and \cite[Algorithm~2]{LuoWuChen2012}). 

The penalty parameter can be updated according to the following standard rule (see 
\cite{BirginMartinez,LuoSunWu,LuoWuChen2012}): 
\begin{equation} \label{eq:StandardAugmLagrPenUpdate}
  c_{n + 1} = \begin{cases}
    c_n, & \text{if } \sigma_n \le \tau \sigma_{n - 1},
    \\
    \gamma c_n, & \text{otherwise}
  \end{cases}
  \quad \forall n \in \mathbb{N},
\end{equation}
for some fixed $\tau \in (0, 1)$ and $\gamma > 1$, where 
\[
  \sigma_n := \frac{1}{c_n} \big\| \Pi_{K^*}(\lambda_n + c_n G(x_n)) - \lambda_n \big\|
\]
for any $n \in \mathbb{N}$ and $\sigma_{-1} := + \infty$. Note that if rule \eqref{eq:StandardAugmLagrPenUpdate} is 
used, then $\sigma_n \to 0$ as $n \to \infty$, if the sequence $\{ c_n \}$ is bounded.

Let us prove convergence of the inexact basic augmented Lagrangian method with the use of the universal convergence
theorem. For the sake of convenience, denote 
\begin{equation} \label{eq:HPRAugmLagr_PenTerm}
\begin{split}
  \varphi(x, (\lambda, c)) 
  &= \Big( \frac{1}{2c} \Big[ \big\| \Pi_{K^*}(\lambda + c G(x)) \big\|^2 - \| \lambda \|^2 \Big]
  \\
  &= \frac{1}{2c} \Big[ \dist\big( \lambda + c G(x), K \big)^2 - \| \lambda \|^2 \Big]
\end{split}
\end{equation}
and $\mu_n = (\lambda_n, c_n)$. 

\begin{theorem} \label{thrm:InexactBasicAugmLagrMethod}
Let $\{ (x_n, \lambda_n, c_n) \}$ be the sequence generated by the inexact basic augmented Lagrangian method. 
Suppose that the sequence $\{ f(x_n) \}$ is bounded below and the two following assumptions hold true:
\begin{enumerate}
\item{if $c_n \to + \infty$ as $n \to \infty$, then 
\begin{equation} \label{eq:MultiplierPenParamQuotient}
  \lim_{n \to \infty} \frac{\| \lambda_n \|}{\sqrt{c_n}} = 0.
\end{equation}
}

\vspace{-5mm}

\item{if the sequence $\{ c_n \}$ is bounded, then (a) the sequence $\{ \lambda_n \}$ is bounded, 
(b) the sequence $\sigma_n := (1/c_n) \| \Pi_{K^*}(\lambda_n + c_n G(x_n)) - \lambda_n \|$, $n \in \mathbb{N}$, 
converges to zero, and (c) either the sequence $\{ \Pi_K(G(x_n)) \}$ is bounded or
$Y = \mathbb{R}^{m + \ell}$ and $K = \{ 0 \} \times \mathbb{R}^{\ell}_-$, where $\mathbb{R}_- = (- \infty, 0]$.
}
\end{enumerate}
Then
\begin{gather*}
  \lim_{n \to \infty} \dist(G(x_n), K) = 0, \quad 
  \limsup_{n \to \infty} \varphi(x_n, (\lambda_n, c_n)) \le \varepsilon_*,
  \\
  \vartheta_* \le \liminf_{n \to \infty} f(x_n) \le \limsup_{n \to \infty} f(x_n) \le \vartheta_* + \varepsilon_*,
  \\
  \vartheta_* - \varepsilon_* \le \liminf_{n \to \infty} \Theta(\lambda_n, c_n) 
  \le \limsup_{n \to \infty} \Theta(\lambda_n, c_n) \le \vartheta_*,
\end{gather*}
where $\vartheta_* := \min\{ f_*, \liminf_{y \to 0} \beta(y) \}$. If, in addition, the optimal value function $\beta$
is lsc at the origin, then $\vartheta_* = f_*$ and the following statements hold true:
\begin{enumerate}
\item{if $f$ is lsc on $Q$ and $G$ is continuous on $Q$, then any limit point of the sequence $\{ x_n \}$ (if such point
exists) is an $\varepsilon_*$-optimal solution of problem \eqref{prob:ConeConstrained};
}

\item{if the sequence of penalty parameters $\{ c_n \}$ is bounded, then any limit point of the sequence
$\{ (\lambda_n, c_n) \}$ (if such point exists) is an $\varepsilon_*$-optimal solution of the dual problem.
} 
\end{enumerate}
\end{theorem}

\begin{proof}
To apply the universal convergence theorem, define $M = Y \times (0, + \infty)$ and 
$F(x, \mu) = f(x) + \varphi(x, \mu)$ with $\mu = (\lambda, c)$. Our aim is to show that
\begin{equation} \label{eq:ConvCond_BasicAugmLagrMethod}
  \lim_{n \to \infty} \dist(0, \Phi(x_n)) = 0, \quad \liminf_{n \to \infty} \varphi(x_n, \mu_n) \ge 0,
\end{equation}
the weak duality holds true, and the sequence $\{ (\lambda_n, c_n) \}$ satisfies the $\varphi$-conver\-gen\-ce to zero
condition. Then applying the universal convergence theorem for inexact primal-dual methods one obtains the required
result.

\textbf{Part 1.} Observe that for any feasible point $x$ and any $\lambda \in Y$ and $c > 0$ one has
\[
  \dist(\lambda + c G(x), K) = \inf_{y \in K} \| \lambda + c G(x) - y \| 
  \le \| \lambda + c G(x) - c G(x) \| \le \| \lambda \|
\]
since $G(x) \in K$ and $K$ is a cone. Therefore, for any feasible point $x$ and any $\mu = (\lambda, c) \in M$ one has
\[
  \varphi(x, \mu) =  \frac{1}{2c} \Big[ \dist\big( \lambda + c G(x), K \big)^2 - \| \lambda \|^2 \Big]
  \le \frac{1}{2c} \Big[ \| \lambda \|^2 - \| \lambda \|^2 \Big] \le 0
\]
(see \eqref{eq:HPRAugmLagr_PenTerm}), which by Lemma~\ref{lem:WeakDuality} implies that the weak duality holds true. 

\textbf{Part 2.} Let us now check the $\varphi$-convergence to zero condition. Indeed, fix some 
$n \in \mathbb{N}$, $t_n > 0$, and $z_n \in Q$ such that $\dist(G(z_n), K) \le t_n$. Then 
$\dist(c_n G(z_n), K) \le c_n t_n$ and, therefore, one can find a vector $y_n \in K$ such that 
$\| c_n G(z_n) - y_n \| \le 2 c_n t_n$. Consequently, one has
\[
  \dist(\lambda_n + c_n G(z_n), K) \le \| \lambda_n + c_n G(z_n) - y_n \|
  \le \| \lambda_n \| + 2 c_n t_n,
\]
which implies that
\begin{align*}
  \dist(\lambda_n + c_n G(z_n), K)^2 - \| \lambda_n \|^2
  &\le \big( \| \lambda_n \| + 2 c_n t_n \big)^2 - \| \lambda_n \|^2
  \\
  &= 4 c_n t_n \| \lambda_n \| + 4 c_n^2 t_n^2.
\end{align*}
Therefore, choosing $t_n > 0$ in such a way that $2 t_n \| \lambda_n \| + 2 c_n t_n^2 \le 1 / (n + 1)$ one gets that for
any sequence $\{ z_n \} \subset Q$ with $\dist(G(z_n), K) \le t_n$ one has
\[
  \varphi(z_n, \mu_n) \le \frac{1}{n + 1}  \quad \forall n \in \mathbb{N}
\]
(see \eqref{eq:HPRAugmLagr_PenTerm}). Therefore, the sequence $\{ (\lambda_n, c_n) \}$ satisfied the 
$\varphi$-convergence to zero condition for any such sequence $\{ t_n \}$.

\textbf{Part 3.} Thus, it remains to check that conditions \eqref{eq:ConvCond_BasicAugmLagrMethod} hold true. Let us 
consider two cases corresponding to the two assumptions of the theorem.

\textbf{Case I.} Let the sequence of penalty parameters $\{ c_n \}$ be bounded. Suppose that $\dist(G(x_n), K) \to 0$
as $n \to \infty$. From \eqref{eq:HPRAugmLagr_PenTerm} it follows that
\[
  \varphi(x_n, \mu_n) = \frac{1}{2 c_n} \Big( \big\| \Pi_{K^*}(\lambda_n + c_n G(x_n)) \big\| - \| \lambda_n \| \Big)   
  \Big( \dist\big( \lambda_n + c_n G(x_n), K \big) + \| \lambda_n \| \Big)  
\]
The sequence $\{ \dist(\lambda_n + c_n G(x_n), K) + \| \lambda_n \| \}$ is bounded due to the boundedness of 
the sequences $\{ \lambda_n \}$ and $\{ c_n \}$ and the fact that $\dist(G(x_n), K) \to 0$ as $n \to \infty$. Hence 
taking into account the facts that
\[
  \frac{1}{2 c_n} \Big| \big\| \Pi_{K^*}(\lambda_n + c_n G(x_n)) \big\| - \| \lambda_n \| \Big|
  \le \frac{1}{2 c_n} \Big\| \Pi_{K^*}(\lambda_n + c_n G(x_n)) - \lambda_n \Big\|
  = \frac{1}{2} \sigma_n
\]
and the sequence $\{ \sigma_n \}$ converges to zero by our assumption, one can conclude that 
$\varphi(x_n, \mu_n) \to 0$ as $n \to \infty$. Therefore it remains to check that $\dist(G(x_n), K) \to 0$ as 
$n \to \infty$. Let us consider two subcases corresponding to the two different assumptions.

\textbf{Case I.A.} Suppose that the sequence $\{ \Pi_K(G(x_n)) \}$ is bounded. From the boundedness of the sequence 
$\{ c_n \}$ and the convergence of the sequence $\{ \sigma_n \}$ to zero it follows that
\begin{equation} \label{eq:KappaConvergesToZero}
  \varkappa_n := \| \Pi_{K^*}(\lambda_n + c_n G(x_n)) - \lambda_n \| \xrightarrow[n \to \infty]{} 0.
\end{equation}
By the definition of projection for any $n \in \mathbb{N}$ and $z \in K^*$ one has
\begin{align*}
  \| c_n G(x_n) \| &= \big\| (\lambda_n + c_n G(x_n)) - \lambda_n \big\|
  \\
  &\le \Big\| \lambda_n + c_n G(x_n) - \Pi_{K^*}(\lambda_n + c_n G(x_n)) \Big\| + \varkappa_n
  \\
  &\le \big\| \lambda_n + c_n G(x_n) - z \big\| + \varkappa_n.
\end{align*}
Since $K^*$ is a convex cone and $\{ \lambda_n \} \subset K^*$, for any $n \in \mathbb{N}$ and $z \in K^*$ one has
$\lambda_n + z \in K^*$. Therefore for any $n \in \mathbb{N}$ and $z \in K^*$ one has
\[
  \| c_n G(x_n) \| \le \big\| \lambda_n + c_n G(x_n) - (\lambda_n + z) \big\| + \varkappa_n
  = \big\| c_n G(x_n) - z \big\| + \varkappa_n,
\]
which yields
\[
  \| G(x_n) \| \le \dist(G(x_n), K^*) + \frac{\varkappa_n}{c_n} = \| \Pi_K(G(x_n)) \| + \frac{\varkappa_n}{c_n}
  \quad \forall n \in \mathbb{N}.
\]
Hence by applying Moreau's decomposition theorem for cones \cite{Moreau,Soltan} one obtains that
\begin{multline*}
  \| \Pi_K(G(x_n)) \|^2 + \| \Pi_{K^*}(G(x_n)) \|^2 = \| G(x_n) \|^2
  \\
  \le \| \Pi_K(G(x_n)) \|^2 + 2 \frac{\varkappa_n}{c_n} \| \Pi_K(G(x_n)) \| + \frac{\varkappa_n^2}{c_n^2}
\end{multline*}
for all $n \in \mathbb{N}$ or, equivalently,
\[
  \| \Pi_{K^*}(G(x_n)) \|^2 \le 2 \frac{\varkappa_n}{c_n} \| \Pi_K(G(x_n)) \| + \frac{\varkappa_n^2}{c_n^2},
\]
which with the use of \eqref{eq:KappaConvergesToZero} and the boundedness of the sequences $\{ \Pi_K(G(x_n)) \}$ and 
$\{ c_n \}$ implies that $\dist(G(x_n), K) = \| \Pi_{K^*}(G(x_n)) \| \to 0$ as $n \to \infty$.

\textbf{Case I.B.} Suppose that $Y = \mathbb{R}^{m + \ell}$ and $K = \{ 0 \} \times \mathbb{R}^{\ell}_-$. Then 
$K^* = \mathbb{R}^m \times \mathbb{R}^{\ell}_+$ and for any $n \in \mathbb{N}$ one has
\begin{multline*}
  \sigma_n = \frac{1}{c_n} \Big\| \Pi_{K^*}(\lambda_n + c_n G(x_n)) - \lambda_n \Big\| =
  \\
  \bigg\| \Big( G_1(x_n), \ldots, G_m(x_n), \max\left\{ G_{m + 1}(x_n), - \frac{1}{c_n} (\lambda_n)_{m + 1} \right\}, 
  \ldots,
  \\
  \max\left\{ G_{m + \ell}(x_n), - \frac{1}{c_n} (\lambda_n)_{m + \ell} \right\} \bigg\|.
\end{multline*}
Combining the equality above with the facts that $\sigma_n \to 0$ as $n \to \infty$ and $\lambda_n \in K^*$ one can
readily verify that $\dist(G(x_n), K) \to 0$ as $n \to \infty$.

\textbf{Case II.} Suppose that $c_n \to + \infty$ as $n \to \infty$ and equality \eqref{eq:MultiplierPenParamQuotient} 
holds true. Recall that
\[
  \varphi\big( x, (\lambda, c) \big) 
  = \inf_{y \in K - G(x)} \left( - \langle \lambda, y \rangle + \frac{c}{2} \| y \|^2 \right)
\]
(see \cite[Section~11.K]{RockafellarWets} and \cite{ShapiroSun}). For any $y \in Y$ one has
\[
  - \langle \lambda, y \rangle + \frac{c}{2} \| y \|^2 \ge - \| \lambda \| \| y \| + \frac{c}{2} \| y \|^2
  \ge \inf_{t \ge 0} \left( - \| \lambda \| t + \frac{c}{2} t^2 \right) = - \frac{\| \lambda \|^2}{2 c}.
\]
Therefore 
\[
  \varphi(x_n, \mu_n) \ge - \frac{\| \lambda_n \|^2}{2 c_n} \quad \forall n \in \mathbb{N}
\]
and $\liminf_{n \to \infty} \varphi(x_n, \mu_n) \ge 0$ due to equality \eqref{eq:MultiplierPenParamQuotient}.

Let us now show that $\dist(G(x_n), K) \to 0$ as $n \to \infty$. Suppose by contradiction that this statement is false.
Then there exist $\delta > 0$ and a subsequence $\{ x_{n_k} \}$ such that $\dist(G(x_{n_k}), K) \ge \delta$ 
for all $k \in \mathbb{N}$. Taking into account \eqref{eq:MultiplierPenParamQuotient} one can suppose that
\[
  \dist\left(\frac{\lambda_{n_k}}{c_{n_k}} + G(x_{n_k}), K \right) \ge \frac{\delta}{2} 
  \quad \forall k \in \mathbb{N}.
\]
Hence bearing in mind \eqref{eq:HPRAugmLagr}, for any $k \in \mathbb{N}$ one gets that
\begin{align*}
  \mathscr{L}(x_{n_k}, \lambda_{n_k}, c_{n_k}) &= f(x_{n_k}) 
  + \frac{1}{2c_{n_k}} \Big[ \dist\big( \lambda_{n_k} + c_{n_k} G(x_{n_k}), K \big)^2 - \| \lambda_{n_k} \|^2 \Big]
  \\
  &= f(x_{n_k}) + \frac{c_{n_k}}{2} \dist\Big(\frac{\lambda_{n_k}}{c_{n_k}} + G(x_{n_k}), K \Big)^2
  - \frac{\| \lambda_{n_k} \|^2}{2 c_{n_k}}
  \\
  &\ge f(x_{n_k}) + \frac{c_{n_k}}{4} \delta - \frac{\| \lambda_{n_k} \|^2}{2 c_{n_k}}.
\end{align*}
Consequently, $\mathscr{L}(x_{n_k}, \lambda_{n_k}, c_{n_k}) \to + \infty$ as $k \to \infty$ due to equality 
\eqref{eq:MultiplierPenParamQuotient} and the boundedness below of the sequence $\{ f(x_{n_k}) \}$, which contradicts
the fact that $\mathscr{L}(x_{n_k}, \lambda_{n_k}, c_{n_k}) \le f_* + \varepsilon_{n_k}$ for all $k \in \mathbb{N}$ due 
to the weak duality. Therefore, $\dist(G(x_n), K) \to 0$ as $n \to \infty$.

Thus, in both cases conditions \eqref{eq:ConvCond_BasicAugmLagrMethod} are satisfied and applying 
Theorem~\ref{thrm:UnivConvThrmInexact} we arrive at the required results.
\end{proof}

\begin{remark}
Let us briefly discuss the assumptions of Theorem~\ref{thrm:InexactBasicAugmLagrMethod}:
\begin{enumerate}
\item{In the case of inequality constrained problems, condition~\eqref{eq:MultiplierPenParamQuotient} always holds true,
if one uses the so-called \textit{conditional multiplier updating} (see e.g. \cite{LuoSunLi,LuoSunWu}). An extension of
the conditional multplier updating to the case of general cone constrained problems is possible and is an interesting
topic for future research.
}

\item{The assumption on the boundedness of the sequence of multipliers $\{ \lambda_n \}$ is common in the literature on
augmented Lagrangian methods \cite{Bertsekas,BirginMartinez,LuoSunLi,LuoSunWu,LuoWuChen2012}. One can ensure the
boundedness of this sequence by using the multiplier method with safeguarding \eqref{eq:MultUpdated_Safeguarding}.
}

\item{The sequence $\{ \sigma_n \}$ convergences to zero whenever the sequence of penalty parameters $\{ c_n \}$ is
bounded, if the standard rule \eqref{eq:StandardAugmLagrPenUpdate} for updating the penalty parameter is used.
}

\item{The assumption on the boundedness of the sequence $\{ \Pi_K(G(x_n)) \}$ is satisfied, provided the sequence 
$\{ x_n \}$ is bounded and $G$ maps bounded sets into bounded sets. In particular, if $X$ is finite dimensional, it is 
sufficient to suppose that the sequence $\{ x_n \}$ is bounded and the function $G$ is continuous.}

\item{Note that problem \eqref{prob:ConeConstrained} with $Y = \mathbb{R}^{m + \ell}$ and 
$K = \{ 0 \} \times \mathbb{R}^{\ell}_-$ is just a standard mathematical programming problems with $m$ equality and
$\ell$ inequality constraints.}
\end{enumerate}
\end{remark}

\begin{remark}
Theorem~\ref{thrm:InexactBasicAugmLagrMethod} unifies, significantly generalizes, and strengthens many existing
results on converges of primal-dual augmented Lagrangian methods based on the Hestenes-Powell-Rockafellar augmented
Lagrangian. In particular, \cite[Theorem~5.2]{BirginMartinez}, \cite[Theorem~2.4 and 3.1]{LuoSunWu} in the case when
$\sigma(z) \equiv 0.5 z^2$, and \cite[Theorems~1, 3, and 4]{LuoWuChen2012} follow directly from
Theorem~\ref{thrm:InexactBasicAugmLagrMethod}. Moreover, Theorem~\ref{thrm:InexactBasicAugmLagrMethod} strengthens 
the corresponding theorems from \cite{LuoSunWu,LuoWuChen2012}, since we do not impose the assumption that the objective
function $f$ is bounded below on $Q$ and only suppose that the sequence $\{ f(x_n) \}$ is bounded below.
\end{remark}

\subsection{Multiplier algorithm based on PALM}
\label{subsect:PALMAlgorithm}

Let us finally apply the universal convergence theorem for inexact primal-dual methods to the multiplier algorithm
based on the so-called P-type Augmented Lagrangian Method (PALM) that was proposed in \cite{WangLi2009}. One can prove 
convergence of all other augmented Lagrangian methods from \cite{WangLi2009} with the use of the universal convergence
theorem in exactly the same way.

In this section we study an augmented Lagrangian method for the inequality constrained problem:
\begin{equation} \label{prob:InequalConstr}
  \min \: f(x) \quad \text{subject to} \quad g(x) \le 0, \quad x \in Q,
\end{equation}
where $g \colon X \to \mathbb{R}^{\ell}$ is a given function. Fix some $a \in (0, + \infty]$, and let a function
$P \colon (- \infty, a) \times \mathbb{R}_+ \to \mathbb{R}$, $P = P(s, t)$, be continuous in $(s, t)$ and continuously 
differentiable in $s$. Suppose also that for some continuous function $r \colon \mathbb{R}_+ \to \mathbb{R}$ one has
\begin{equation} \label{PFuncCond}
  P(s, t) \ge r(t) \quad \forall s \in (- \infty, a), \: t \in \mathbb{R}_+.
\end{equation}
Following \cite{WangLi2009}, we call the function
\[
  L_P(x, \lambda, c) = \begin{cases}
    f(x) + \frac{1}{c} \sum_{i = 1}^{\ell} P(c g_i(x), \lambda_i), & \text{if }
    c g_i(x) < a \enspace \forall i \in \{ 1, \ldots, \ell \},
    \\
    + \infty, & \text{otherwise}
  \end{cases}
\]
\textit{the P-type augmented Lagrangian} for problem \eqref{prob:InequalConstr}. Here 
$\lambda = (\lambda_1, \ldots, \lambda_{\ell}) \in \mathbb{R}^{\ell}_+$ are multipliers and $c > 0$ is the penalty 
parameter. Several particular examples of the function $P$ can be found in \cite[Section~2]{WangLi2009}. Here we only
note that the well-known Hestenes-Powell-Rockafellar augmented Lagrangian \cite{Rockafellar73,BirginMartinez} is a
particular case of the P-type augmented Lagrangian.

In \cite[Section~4]{WangLi2009} the following primal-dual method, given in Algorithm~\ref{alg:PALM} and called 
\textit{the multiplier algorithm based on PALM}, was proposed.

\begin{algorithm}
\caption{Multiplier algorithm based on PALM}

\textbf{Initialization.} Choose initial points $x_0 \in Q$, $\lambda_0 \ge 0$, and $c_0 > 0$, and a bounded above
sequence of tolerances $\{ \varepsilon_n \} \subset [0, + \infty)$. Put $n = 0$.

\textbf{Step 1. Update of parameters.} Compute
\begin{align*} 
  \lambda_{(n + 1) i} &= P'_s(c_n g_i(x_n), \lambda_{ni}), \quad i \in \{ 1, \ldots, \ell \},
  \\
  c_{n + 1} &\ge (n + 1) \max\Big\{ 1, \sum_{i = 1}^{\ell} |r(\lambda_{(n + 1) i})| \Big\},
\end{align*}
where $P'_s$ is the derivative of the function $s \mapsto P(s, t)$.

\textbf{Step 2. Solution of subproblem.} Find an $\varepsilon_{n + 1}$-optimal solution $x_{n + 1}$ of the problem
\[
  \min_x \: L_P(x, \lambda_n, c_n) \quad \text{subject to} \quad x \in Q.
\]
Increment $n$ and go to \textbf{Step 1}.
\label{alg:PALM}
\end{algorithm}

Let us analyse convergence of Algorithm~\ref{alg:PALM} with the use of the universal convergence theorem for
inexact primal-dual methods. 

\begin{theorem}
Let $\{ (x_n, \lambda_n, c_n) \}$ be the sequence generated by Algorithm~\ref{alg:PALM}. Suppose that the sequence
$\{ f(x_n) \}$ is bounded below and the function $P$ satisfies the following assumptions:
\begin{enumerate}
\item{the function $s \mapsto P(s, t)$ is nondecreasing and $P(0, t) = 0$ for any $t \ge 0$;}

\item{if $a = + \infty$, then $P(s, t) / s \to + \infty$ as $s \to + \infty$ for any $t \ge 0$.}
\end{enumerate}
Then
\begin{gather} \label{eq:PALM_InfMeasLim}
  \lim_{n \to \infty} \sum_{i = 1}^{\ell} \max\{ 0, g_i(x_n) \} = 0,
  \\ \label{eq:PALM_ObjFuncInexactLim}
  \vartheta_* \le \liminf_{n \to \infty} f(x_n) \le \limsup_{n \to \infty} f(x_n) \le \vartheta_* + \varepsilon_*,
  \\ \label{eq:PALM_DualFuncInexactLim}
  \vartheta_* - \varepsilon_* \le \liminf_{n \to \infty} \Theta(\mu_n)
  \le \limsup_{n \to \infty} \Theta(\mu_n) \le \vartheta_*,
\end{gather}
where $\vartheta_* = \min\{ f_*, \liminf_{y \to 0} \beta(y) \}$. If, in addition, the optimal value function $\beta$ is
lsc at the origin and the functions $f$ and $g_i$, $i \in \{ 1, \ldots, \ell \}$ are lsc on $Q$, then 
$\vartheta_* = f_*$ and limit points of the sequence $\{ x_n \}$ (if such points exist) are $\varepsilon_*$-optimal
solutions of problem \eqref{prob:InequalConstr}.
\end{theorem}

\begin{proof}
It should be noted that the first assumption on the function $P$ ensures that $\lambda_n \ge 0$ for any
$n \in \mathbb{N}$ and, therefore, the method is correctly defined.

Denote $M = \mathbb{R}^{\ell}_+ \times (0, + \infty)$ and for any $\mu = (\lambda, c) \in M$ define
\[
  \varphi(x, \mu) = \begin{cases}
    \frac{1}{c} \sum_{i = 1}^{\ell} P(c g_i(x), \lambda_i), & \text{if }
    c g_i(x) < a \enspace \forall i \in \{ 1, \ldots, \ell \},
    \\
    + \infty, & \text{otherwise.}
  \end{cases}
\]
Then $L_P(x, \lambda, c) = f(x) + \varphi(x, (\lambda, c))$. For the sake of convenience, denote 
$\mu_n = (\lambda_n, c_n)$.

By the first assumption on the function $P$ for any $s \le 0$ and $t \in \mathbb{R}_+$ one has $P(s, t) \le 0$, which 
implies that for any feasible point $x$ and any $\mu \in M$ one has $\varphi(x, \mu) \le 0$. Hence by 
Lemma~\ref{lem:WeakDuality} the weak duality holds and, therefore, $L_P(x_n, \lambda_n, c_n) \le f_* + \varepsilon_n$
for any $n \in \mathbb{N}$.

Bearing in mind condition \eqref{PFuncCond} and the rule for updating the penalty parameter on Step~1 of 
Algorithm~\ref{alg:PALM} one gets that
\[
  \varphi(x_n, \mu_n) = \frac{1}{c_n} \sum_{i = 1}^{\ell} P(c_n g_i(x_n), \lambda_{ni})
  \ge \frac{1}{c_n} \sum_{i = 1}^{\ell} r(\lambda_{ni}) \ge - \frac{1}{c_n} \sum_{i = 1}^{\ell} |r(\lambda_{ni})|
  \ge - \frac{1}{n}
\]
for any $n \ge 1$, which implies that $\liminf_{n \to \infty} \varphi(x_n, \mu_n) \ge 0$.

Let us now check that $\max\{ 0, g_i(x_n) \} \to 0$ as $n \to \infty$ for any $i \in \{ 1, \ldots, \ell \}$. 
Suppose by contradiction that this statement is false. Then there exist $i \in \{ 1, \ldots, \ell \}$, $\delta > 0$, 
and a subsequence $\{ x_{n_k} \}$ such that $g_i(x_{n_k}) \ge \delta$. 

By the definition of $c_n$ (see Step~1 of Algorithm~\ref{alg:PALM}) one has $c_n \ge n$ for all $n \in \mathbb{N}$.
Consequently, if $a < + \infty$, then $c_{n_k} g_i(x_{n_k}) \ge a$ and $\varphi(x_{n_k}, \mu_{n_k}) = + \infty$ for any
sufficiently large $k$, which is impossible.

Let us now consider the case $a = + \infty$. With the use of condition \eqref{PFuncCond} and the fact that the function 
$s \mapsto P(s, t)$ is nondecreasing one gets
\begin{align*}
  L_P(x_{n_k},& \lambda_{n_k}, c_{n_k})
  = f(x_{n_k}) + \frac{1}{c_{n_k}} \sum_{j = 1}^{\ell} P(c_{n_k} g_j(x_{n_k}), \lambda_{{n_k}j})
  \\
  &\ge \inf_n f(x_n) + \frac{1}{c_{n_k}} P(c_{n_k} g_i(x_{n_k}), \lambda_{{n_k}i})
  + \frac{1}{c_{n_k}} \sum_{j \ne i} P(c_{n_k} g_j(x_{n_k}), \lambda_{{n_k}j}) 
  \\
  &\ge \inf_n f(x_n) + \frac{1}{c_{n_k}} P(c_{n_k} \delta, \lambda_{{n_k}i})
  - \frac{1}{c_{n_k}} \sum_{j \ne i} |r(\lambda_{{n_k}j})|
  \\
  &\ge \inf_n f(x_n) + \frac{1}{c_{n_k}} P(c_{n_k} \delta, \lambda_{{n_k}i}) - \frac{1}{n_k}.
\end{align*}
By our assumption the sequence $\{ f(x_n) \}$ is bounded below. Hence taking into account the second assumption on 
the function $P$ one can conclude that $L_P(x_{n_k}, \lambda_{n_k}, c_{n_k}) \to + \infty$ as $k \to \infty$, which
contradicts the fact that by the weak duality $L_P(x_n, \lambda_n, c_n) \le f_* + \varepsilon_n$ for any 
$n \in \mathbb{N}$. Thus, $\max\{ 0, g_i(x_n) \} \to 0$ as $n \to \infty$ for any $i \in \{ 1, \ldots, \ell \}$.

Let us finally verify that the sequence $\{ (\lambda_n, c_n) \}$ satisfies the $\varphi$-convergence to zero condition.
Indeed, since $P(0, t) = 0$ and the function $s \mapsto P(s, t)$ is continuous and nondecreasing for all 
$t \in \mathbb{R}_+$, for any $n \ge 1$ one can find $t_n > 0$ such that $c_n s < a$ and 
$P(c_n s, \lambda_{ni}) < c_n / n \ell$ for all $s \le t_n$. Therefore, for any sequence $\{ z_n \} \subset Q$
satisfying the inequality $\sum_{i = 1}^{\ell} \max\{ 0, g_i(z_n) \} \le t_n$ one has
\[
  \varphi(z_n, \mu_n) = \frac{1}{c_n} \sum_{i = 1}^{\ell} P(c_n g_i(z_n), \lambda_{ni}) \le \frac{1}{n}
\]
for any $n \in \mathbb{N}$. Consequently, $\limsup_{n \to \infty} \varphi(z_n, \mu_n) \le 0$ and the sequence 
$\{ (\lambda_n, c_n) \}$ satisfies the $\varphi$-convergence to zero condition. Therefore, applying
Theorem~\ref{thrm:UnivConvThrmInexact} we arrive at the required result.
\end{proof}

\begin{remark}
The previous theorem significantly generalizes \cite[Theorem~5]{WangLi2009}, since it shows that 
\cite[Assumption~$(H_4)$]{WangLi2009} and the assumptions that
\[
  P(s, 0) \ge 0 \quad \forall s \in (- \infty, a), \quad
  \lim_{t \to + \infty} P(s, t) = + \infty \quad \forall s > 0
\]
are completely redundant, while the assumption on the boundedness below of the function $f$ on the set $Q$
(see \cite[Assumption~1]{WangLi2009}) can be replaced by a much less restrictive assumption on the boundedness below 
of the sequence $\{ f(x_n) \}$.
\end{remark}

\section{Concluding remarks}

We presented the universal convergence theorem for inexact primal-dual methods that (i) can be applied to a wide variety
of penalty and augmented Lagrangian methods and (ii) reduces convergence analysis of these methods to verification of
some simple conditions related to the convergence of the infeasibility measure to zero. By applying this theorem to
convergence analysis of several particular primal-dual methods, we demonstrated that with the use of the universal
convergence theorem one can not only easily recover existing results on convergence of various primal-dual methods, but
also significantly strengthen and generalize them.

Although the universal convergence theorem is a very general result that can be applied to, perhaps, a majority of
existing primal-dual penalty and augmented Lagrangian methods, there are some methods for which one cannot recover
existing convergence theorems with the use of the universal convergence theorem. The modified subgradient algorithm from
\cite{Gasimov} is one such method. One can prove its dual convergence, despite the fact that the infeasibility measure
might not converge to zero along the sequences generated by this method (see \cite[Example~1]{BurachikGasimov}).
Moreover, for some closely related primal-dual methods
\cite{BurachikKayaMammadov,BurachikIusemMelo2013,BurachikLiu2023} one can prove convergence theorems whose claims are
significantly stronger than the claim of the universal convergence theorem (e.g. one can prove strong convergence of
the dual variables). 

Thus, although the universal convergence theorem unifies and significantly generalizes many existing results on
convergence of primal-dual methods, it should be viewed not as a universal result that can completely eliminate 
the need for detailed convergence analysis, but rather as a useful tool that can significantly simplify this analysis
in many important cases.

\bibliographystyle{abbrv}  
\bibliography{UnivConvergenceThm_bibl}

\begin{thebibliography}{10}

\bibitem{Bertsekas}
D.~P. Bertsekas.
\newblock {\em Constrained Optimization and Lagrange Multiplier Methods}.
\newblock Academic Press, New York, 1982.

\bibitem{BirginFloudasMartinez}
E.~G. Birgin, C.~A. Floudas, and J.~M. Martinez.
\newblock Global minimization using an {A}ugmented {L}agrangian method with
  variable lower-level constraints.
\newblock {\em Math. Program.}, 125:139--162, 2010.

\bibitem{BirginMartinez}
E.~G. Birgin and J.~M. Martinez.
\newblock {\em Practical Augmented {L}agrangian Methods for Constrained
  Optimization}.
\newblock SIAM, Philadelphia, 2014.

\bibitem{Burachick2011}
R.~S. Burachik.
\newblock On primal convergence for augmented {L}agrangian duality.
\newblock {\em Optim.}, 60:979--990, 2011.

\bibitem{BurachikGasimov}
R.~S. Burachik, R.~N. Gasimov, N.~A. Ismayilova, and C.~Kaya.
\newblock On a modified subgradient algorithm for dual problems via sharp
  augmented {L}agrangian.
\newblock {\em J. Glob. Optim.}, 34:55--78, 2006.

\bibitem{BurachickIusemMelo_SharpLagr}
R.~S. Burachik, A.~N. Iusem, and J.~G. Melo.
\newblock A primal dual modified subgradient algorithm with sharp {L}agrangian.
\newblock {\em J. Glob. Optim.}, 46:55--78, 2010.

\bibitem{BurachikIusemMelo2013}
R.~S. Burachik, A.~N. Iusem, and J.~G. Melo.
\newblock An inexact modified subgradient algorithm for primal-dual problems
  via augmented {L}agrangians.
\newblock {\em J. Optim. Theory Appl.}, 157:108--131, 2013.

\bibitem{BurachikKayaMammadov}
R.~S. Burachik, C.~Y. Kaya, and M.~Mammadov.
\newblock An inexact modified subgradient algorithm for nonconvex optimization.
\newblock {\em Comput. Optim. Appl.}, 45:1--34, 2008.

\bibitem{BurachikKayaPrice}
R.~S. Burachik, C.~Y. Kaya, and C.~J. Price.
\newblock A primal-dual penalty method via rounded weighted-$\ell_1$
  {L}agrangian duality.
\newblock {\em Optim.}, 71:3981--4017, 2022.

\bibitem{BurachikLiu2023}
R.~S. Burachik and X.~Liu.
\newblock An inexact deflected subgradient algorithm in infinite dimensional
  spaces.
\newblock {\em arXiv: 2302.02072}, 2023.

\bibitem{CordovaOliveiraSagastizabal}
M.~Cordova, W.~Oliveira, and C.~Sagastiz\'{a}bal.
\newblock Revisiting augmented {L}agrangian duals.
\newblock {\em Math. Program.}, 196:235--277, 2022.

\bibitem{CuiDingLiZhao}
Y.~Cui, C.~Ding, X.~Li, and X.~Zhao.
\newblock Augmented {L}agrangian method for convex matrix optimization
  problems.
\newblock {\em J. Oper. Res. Soc. China}, 10:305--342, 2022.

\bibitem{Demyanov2010}
V.~F. Demyanov.
\newblock Nonsmooth optimization.
\newblock In G.~{D}i Pillo and F.~Schoen, editors, {\em Nonlinear optimization.
  Lecture notes in mathematics, vol. 1989}, pages 55--163. Springer-Verlag,
  Berlin, 2010.

\bibitem{Dolgopolik2016}
M.~V. Dolgopolik.
\newblock A unifying theory of exactness of linear penalty functions.
\newblock {\em Optim.}, 65:1167--1202, 2016.

\bibitem{Dolgopolik2017}
M.~V. Dolgopolik.
\newblock A unifying theory of exactness of linear penalty functions {I}{I}:
  parametric penalty functions.
\newblock {\em Optim.}, 66:1577--1622, 2017.

\bibitem{Dolgopolik2022}
M.~V. Dolgopolik.
\newblock Exact penalty functions with multidimensional penalty parameter and
  adaptive penalty updates.
\newblock {\em Optim. Lett.}, 16:1281--1300, 2022.

\bibitem{Dolgopolik_DCSemidef}
M.~V. Dolgopolik.
\newblock {D}{C} semidefinite programming and cone constrained {D}{C}
  optimization {I}{I}: local search methods.
\newblock {\em Comput. Optim. Appl.}, 85:993--1031, 2023.

\bibitem{Dolgopolik2024}
M.~V. Dolgopolik.
\newblock Convergence analysis of primal-dual augmented {L}agrangian methods
  and duality theory.
\newblock {\em arXiv: 2409.13974}, 2024.

\bibitem{Gasimov}
R.~N. Gasimov.
\newblock Augmented {L}agrangian duality and nondifferentiable optimization
  methods in nonconvex programming.
\newblock {\em J. Glob. Optim.}, 24:187--203, 2002.

\bibitem{Grossmann}
C.~Grossmann.
\newblock Smoothing techniques for exact penalty methods.
\newblock In C.~M. {da Fonseca}, D.~V. Huynh, S.~Kirkland, and V.~K. Tuan,
  editors, {\em Contemporary Mathematics. Vol. 658. A Panorama of Mathematics:
  Pure and Applied}, pages 249--265. Americal Mathematical Society, Providence,
  Rhode Island, 2016.

\bibitem{Iusem99}
A.~N. Iusem.
\newblock Augmented {L}agrangian methods and proximal point methods for convex
  optimization.
\newblock {\em Investigaci\'{o}n Operativa}, 8:11--49, 1999.

\bibitem{LiuZhang2007}
Y.~J. Liu and L.~W. Zhang.
\newblock Convergence analysis of the augmented {L}agrangian method for
  nonlinear second-order cone optimization problems.
\newblock {\em Nonlinear Anal.: Theory, Methods, Appl.}, 67:1359--1373, 2007.

\bibitem{LiuzziLucidi}
G.~Liuzzi and S.~Lucidi.
\newblock A derivative-free algorithm for inequality constrained nonlinear
  programming via smoothing of an $\ell_{\infty}$ penalty function.
\newblock {\em SIAM J. Optim.}, 20:1--29, 2009.

\bibitem{LuoSunWu}
H.~Luo, X.~Sun, and H.~Wu.
\newblock Convergence properties of augmented {L}agrangian methods for
  constrained global optimization.
\newblock {\em Optim. Methods Softw.}, 23:763--778, 2008.

\bibitem{LuoSunLi}
H.~Z. Luo, X.~L. Sun, and D.~Li.
\newblock On the convergence of augmented {L}agrangian methods for constrained
  global optimization.
\newblock {\em SIAM J. Optim.}, 18:1209--1230, 2007.

\bibitem{LuoWuChen2012}
H.~Z. Luo, H.~X. Wu, and G.~T. Chen.
\newblock On the convergence of augmented {L}agrangian methods for nonlinear
  semidefinite programming.
\newblock {\em J. Glob. Optim.}, 54:599--618, 2012.

\bibitem{MengDangYang}
Z.~Meng, C.~Dang, and X.~Q. Yang.
\newblock On the smoothing of the square-root exact penalty function for
  inequality constrained optimization.
\newblock {\em Comput. Optim. Appl.}, 35:375--398, 2006.

\bibitem{Moreau}
J.~J. Moreau.
\newblock D\'{e}composition orthogonale d'un espace hilbertien selon deux
  c{\^o}nes mutuellement polaires.
\newblock {\em Competes rendus hebdomadaires des s\'{e}ances de
  l'{A}cad\'{e}mie des sciences}, 255:238--240, 1962.

\bibitem{Pinar}
M.~C. Pinar and S.~A. Zenios.
\newblock On smoothing exact penalty functions for convex constrained
  optimization.
\newblock {\em SIAM J. Optim.}, 4:486--511, 1994.

\bibitem{Polyak2002}
R.~A. Polyak.
\newblock Nonlinear rescaling vs. smoothing technique in convex optimization.
\newblock {\em Math. Program.}, 92:197--235, 2002.

\bibitem{Rockafellar73}
R.~T. Rockafellar.
\newblock The multiplier method of {H}estenes and {P}owell applied to convex
  programming.
\newblock {\em J. Optim. Theory Appl.}, 12:555--562, 1973.

\bibitem{RockafellarWets}
R.~T. Rockafellar and R.~J.-B. Wets.
\newblock {\em Variational Analysis}.
\newblock Springer-Verlar, Berlin, 1998.

\bibitem{RubinovHuangYang2002}
A.~M. Rubinov, X.~X. Huang, and X.~Q. Yang.
\newblock The zero duality gap property and lower semicontinuity of the
  perturbation function.
\newblock {\em Math. Oper. Res.}, 27:775--791, 2002.

\bibitem{ShapiroSun}
A.~Shapiro and J.~Sun.
\newblock Some properties of the augmented {L}agrangian in cone constrained
  optimization.
\newblock {\em Math. Oper. Res.}, 29:479--491, 2004.

\bibitem{Soltan}
V.~Soltan.
\newblock Moreau-type characterizations of polar cones.
\newblock {\em Linear Algebra Appl.}, 567:45--62, 2019.

\bibitem{TsengBertsekas}
P.~Tseng and D.~P. Bertsekas.
\newblock On the convergence of the exponential multiplier method for convex
  programming.
\newblock {\em Math. Program.}, 60:1--19, 1993.

\bibitem{WangLi2009}
C.~Y. Wang and D.~Li.
\newblock Unified theory of augmented {L}agrangian methods for constrained
  global optimization.
\newblock {\em J. Glob. Optim.}, 44:433--458, 2009.

\bibitem{WuLuoDingChen2013}
H.~Wu, H.~Luo, X.~Ding, and G.~Chen.
\newblock Global convergence of modified augmented {L}agrangian methods for
  nonlinear semidefinite programming.
\newblock {\em Comput. Optim. Appl.}, 56:531--558, 2013.

\bibitem{Xu2021}
Y.~Xu.
\newblock Iteration complexity of inexact augmented {L}agrangian method for
  constrained convex programming.
\newblock {\em Math. Program.}, 185:199--244, 2021.

\bibitem{YangMengHuang}
X.~Q. Yang, Z.~Q. Meng, X.~X. Huang, and G.~Pong.
\newblock Smoothing nonlinear penalty functions for constrained optimization
  problems.
\newblock {\em Numer. Funct. Anal. Optim.}, 24:351--364, 2003.

\bibitem{Zaslavski}
A.~J. Zaslavski.
\newblock {\em Optimization on Metric and Normed Spaces}.
\newblock Springer, New York, 2010.

\end{thebibliography}

\end{document}